\documentclass[reqno]{amsart}
\usepackage{amsmath}
\usepackage{amsfonts}
\usepackage{amssymb}
\usepackage{amstext}
\usepackage{amsbsy}
\usepackage{amsopn}
\usepackage{amsthm}
\usepackage{amsxtra}
\usepackage{upref}
\usepackage{graphicx}
\usepackage{epstopdf}
\DeclareFontFamily{OML}{rsfs}{\skewchar\font'177}
\DeclareFontShape{OML}{rsfs}{m}{n}{ <5> <6> rsfs5 <7> <8> <9> rsfs7
  <10> <10.95> <12> <14.4> <17.28> <20.74> <24.88> rsfs10 }{}
\DeclareMathAlphabet{\mathfs}{OML}{rsfs}{m}{n}

\newtheorem{theorem}{Theorem}[section]
\newtheorem*{theorem*}{Theorem}
\newtheorem{lemma}{Lemma}[section]
\newtheorem{proposition}{Proposition}[section]

\newtheorem*{example*}{Example}
\newtheorem{definition}{Definition}[section]
\newtheorem{corollary}{Corollary}[section]

\numberwithin{equation}{section}

\renewcommand{\mod}{\mbox{$\,\mathrm{mod}\,$}}

\renewcommand{\epsilon}{\varepsilon}

\newcommand{\bb}[1]{\mathbb{#1}}  
\newcommand{\pg}{P_{G}(\phi)} 
\newcommand{\eqm}{equilibrium measure } 
\newcommand{\ent}[1]{h_{#1}(T)} 
\newcommand{\ass}{Let $X$ be a topologically mixing TMS and let $\phi:X \rightarrow \bb{R}$ be Walters. } 
\newcommand{\vp}{\varphi} 
\newcommand{\Lp}{L_{\phi}} 

\newcommand{\eps}[0]{\epsilon}
\newcommand{\sums}[0]{\sum_{i=0}^{\infty} \hat f(\hat T^i\hat x) -\hat f (\hat T^i\bar x)}

\newcommand{\ua}{\underline a}

\newcommand{\cal}[1]{\mathcal{#1}} 
\newcommand{\Lps}{L_{\phi^*}} 
\newcommand{\la}{\langle}
\newcommand{\ra}{\rangle}
\newcommand{\lar}{\langle \underline a \rangle}
\newtheorem{fact}{Fact}
\newtheorem{claim}{Claim}

\newcommand{\vari}{{\rm var}}
\begin{document}

\title{Bernoullicity of equilibrium measures on countable Markov shifts}
\author{Yair Daon}
\date{\today}

\begin{abstract}
We study the equilibrium behaviour of a two-sided topological Markov shift with a countable number of states. We assume the potential associated with this shift is Walters with finite first variation and that the shift is topologically transitive. We show the equilibrium measure of the system is Bernoulli up to a period. In the process we generalize several theorems on countable Markov shifts. We prove a variational principle and the uniqueness of equilibrium measures. A key step is to show that functions with Walters property on a two-sided shift are cohomologous to one-sided functions with the Walters property. Then we turn to show that functions with summable variations on two-sided CMS are cohomologous to one-sided functions, also with summable variations.
\end{abstract}

\maketitle

\section{Introduction}
\subsection{Countable Markov shifts}
Let $S$ be a finite or countable alphabet. Let $A$  be an $|S| \times |S|$ matrix with entries in $\{0,1\}$. A \textit{(one-sided) topological Markov shift} is a pair $(X,T)$ where $X:= \{x\in S^{\mathbb{N}}|\forall i\in \mathbb{N},\ A_{{x_i}{x_{i+1}}} =1\}$ and $T:X\rightarrow X$, $T(x_0,x_1,x_2,...)=(x_1,x_2,...)$. If $|S| = \aleph_{0}$ we call $X$ a \textit{countable Markov shift} (CMS). $X$ is called a \textit{one-sided shift space}, $T$ is called a \textit{shift operator} and clearly $TX\subseteq X$. The \textit{two-sided shift space} $(\hat X,\hat T)$ is defined similarly except that now $\hat X \subseteq S^{\bb{Z}}$ is made of two-sided sequences and $\hat T$ is a left shift. Objects related to the two-sided shift will have hats: a member of the two sided shift space is, for example $\hat x \in \hat X$. The topology we use is always the product topology induced by the discrete topology on the alphabet $S$. It is metrizable with $d(x,y) := \exp (-\min\{ |i|: x_i\ne y_i\})$ (this applies to both one sided and two sided shift spaces). For the two-sided shift, a basis is defined using \textit{cylinders}: $_m[a_0,...,a_n]:=\{\hat x\in \hat X|x_m=a_0,...,x_{m+n} =a_{n}\}$. Similarly for one-sided shifts: $[a_0,...,a_n]:=\{ x\in X|x_0=a_0,...,x_{n} =a_{n}\}$ (note that in the one-sided case cylinders always start at the zeroth coordinate). A TMS is \textit{topologically mixing} if for every two states $a,b$ there exists $N_{ab}\in \bb{N}$ s.t. $\forall n \geq N_{ab}$ there exist $\xi_i, 1\leq i \leq n-1$ s.t.  $A_{a\xi_1}A_{\xi_1\xi_2}\dots A_{\xi_{{(n-1)}b}} = 1$. A TMS is called \textit{topologically transitive} if for every two states $a,b$, there exists $N:=N_{ab} \in \bb{N}$ there exist $\xi_i, 1\leq i \leq N-1$ s.t.  $A_{a\xi_1}A_{\xi_1\xi_2}\dots A_{\xi_{(N-1)}b} = 1$. Clearly mixing implies transitivity. A fixed real valued function of a shift space (usually referred to as a \textit{potential}) may give rise to equilibrium measures, the analogue of an equilibrium distribution in statistical mechanics.
\subsection{Equilibrium measures}
Let $\phi\in C(X)$ (real-valued continuous). A $T$-invariant Borel probability measure $\mu$ is called an \textit{equilibrium measure}, if it maximizes the quantity $h_{\mu}(T) +\int \phi d\mu$ (subject to the requirement that $h_{\mu}(T) +\int \phi d\mu \ne \infty - \infty$). Equilibrium measures are important because they appear naturally via symbolic dynamics in smooth dynamics (as absolutely continuous invariant measures, physical measures etc.). There is great interest in their ergodic properties. One of the most important tools in studying them is \textit{Ruelle's operator} (a special case of the \textit{transfer operator}), see \cite{sar} for a thorough development of the theory. It is defined for $f:X\to \bb{R}$ as follows: $(L_\phi f)(x):= \sum_{Ty=x} e^{\phi(y)}f(y)$. We'll state the facts we need concerning it as we use them. Ruelle's operator is very useful when working on one-sided shifts, since the term $e^{\phi (y)}$ acts as averaging weights. This operator is less useful in the two-sided invertible case: $\hat T^{-1}\{\hat x\}$ is always a singleton so no averaging is done. Still, there is a way to use this operator on two-sided shifts, as we explain in the next subsection.
\subsection{Cohomology to one-sided function}
Two real-valued functions $f,g$ on a TMS are said to be \textit{cohomologous} if there exists $h$ s.t. $f -g = h -h\circ T$ ($h-h\circ T$ is called a \textit{coboundary} and $h$ is called a \textit{transfer function}). Cohomology is an equivalence relation and it is interesting in the particular case where a two-sided function (i.e. depends on both positive and negative coordinates) is cohomologous to a one-sided function (depends only on non-negative coordinates). We define the \textit{natural projection} $\pi:\hat X \to X, \pi[(x_i)_{i=-\infty}^{\infty}] =(x_i)_{i=0}^{\infty}$.\\
We are interested in the cases where for $\hat f$ (two-sided) there exists $f$ (one-sided) s.t. $\hat f - f\circ \pi = h-h\circ T$. We consider three regularity conditions. Define the $n$th variation of $\hat \phi \in C(\hat X)$ as $\vari_n \hat\phi:= \sup \{ |\hat \phi(x) - \hat\phi(y)|: x_{-n+1}^{n-1} = y_{-n+1}^{n-1}\}$.
\begin{itemize}
\item $\hat \phi$ is \textit {weakly H\"older} if $\exists C>0, 0<\theta<1$ s.t. $\vari_n \hat\phi < C\theta^n$ for $n \geq 2$,
\item $\hat \phi$ has \textit{ summable variations} if $\sum_{n=2}^{\infty} \vari_n \hat\phi <\infty$.
\end{itemize}
We delegate the definition of Walters' condition, to section \ref{waltersection} (definitions \ref{walters1definition}, \ref{walters2definition} there).\\
It is known that H\"older continuity implies summable variations which, in turn, implies Walters condition.
In the finite alphabet case, Sinai \cite{sinai} considered weakly H\"older two-sided functions and showed that each is cohomologous to a one-sided weakly H\"older function (originally appeared in \cite{sinai}, but \cite{bowbook} is more accessible). Coelho \& Quas \cite{coel} did the same for functions with summable variations.  Walters \cite{wal} has done this for functions satisfying Walters condition. All these results, however, were proven in a compact setting. In order to consider infinite alphabet (equivalently, non compact shift spaces), one needs to develop the theory for such spaces. We show that the proof in \cite{sinai} also works for countable alphabet. The proof in \cite{coel} does too, with some modifications. The proof in \cite{wal} relies on a lemma from \cite{bou} which is hard to generalize for a non-compact setting. In this respect, we show Sinai's original construction can be used to find a cohomologous one-sided Walters function to a two-sided Walters function on non compact shift spaces (section \ref{waltersection}).
\subsection{Bernoullicity}
A \textit{Bernoulli scheme} with finite probability vector $(p_{a})_{a\in S}$ is the  left shift $T$ on $S^{\mathbb{Z}}$ with the Borel $\sigma$-algebra $\mathcal{B} (S^{\mathbb{Z}})$ generated by cylinders and $\mu_{p}(_{m}[a_{m},...,a_{n}]):= p_{a_{m}}\hdots p_{a_n}$. Bernoulli schemes are a model of ideal randomness. As such, they were extensively studied. Knowing that a particular system is measure theoretically isomorphic (see \cite{wal2} for definition) to a Bernoulli scheme  gives us complete knowledge of its ergodic properties. We prove isomorphism of equilibrium measures of Walters potential to a product of a Bernoulli schemes and a finite rotation (see theorem \ref{main theorem} below for exact details). 
\subsection{Results}
Our main result is the following theorem.
\begin{theorem}\label{main theorem}
Let $\hat \mu$ be an equilibrium measure of a Walters potential $\hat f\in C(\hat X)$ with finite first variation ($\vari_1 \hat f <\infty$) on a two-sided topologically transitive CMS. Assume $\sup \hat f,h_{\hat \mu}(\hat T)< \infty $. Then $(\hat X,\hat {\mathcal{B}}, \hat \mu ,\hat T)$ is measure theoretically isomorphic to the product of a Bernoulli scheme and a finite rotation.
\end{theorem}
Note that if we assume $\sup \hat f < \infty$ then  $h_{\hat \mu}(\hat T)<\infty$ is equivalent to having finite Gurevich pressure, $P_G(\hat f) <\infty$ (see section \ref{GRPF} for definition, this follows instantly from the variational principle, theorem \ref{variational principle}). Results similar to ours can be found in \cite{bow}, \cite{walreverse}, \cite{sar2}, \cite{berbee} and \cite{rat}. Our results assume very little - we only assume our potential is Walters with finite first variation (as opposed to summable variations in \cite{sar2}). We do not assume compactness, as opposed to \cite{walreverse}. We use a different conditions than \cite{berbee}.\\
We prove that every two-sided Walters potential is cohomologous to a one-sided potential. This is theorem \ref{waltercohomology}. We were also able to prove a similar result for potentials with summable variations, following \cite{coel} (we show the compactness assumed there can be removed). This is theorem \ref{coelhotheorem}.\\
We also prove that an equilibrium measure, if exists, is unique. This was proved in \cite{BuzSar} for summable variations potentials (see section \ref{uniqueness} for precise statement). Also, we prove a variational principle for Walters potentials on non compact (i.e. countable) TMS. This is theorem \ref{variational principle}.
\subsection{Main idea and organization of the proof}
To prove theorem \ref{main theorem} we go through several steps. First, we show we may restrict ourselves to to topologically mixing TMS. In this case we show isomorphism to a Bernoulli scheme (without the finite rotation factor, this is theorem \ref{reduced theorem}). The reduction is stated and proved in section \ref{transection}, using the spectral decomposition. From there on we only concern ourselves with the reduced case of topologically mixing CMS. In section \ref{waltersection} we prove that functions that are Walters with finite first variation are cohomologous to one-sided Walters functions. Section \ref{GRPF} presents the machinery that is used in section \ref{uniqueness}. There, the uniqueness of equilibrium measures is established (theorem \ref{uniqueness theorem}). What we actually need is corollary \ref{corollary2uniqueness}. This corollary gives us important information on equilibrium measures for one-sided shift spaces (of course, with a corresponding one-sided potential). The key is to understand how we can relate two-sided equilibrium measures to one-sided equilibrium measures. This is explained in the beginning of section \ref{Ornstein} (the original idea is due to Sinai, \cite{sinai}). Having established the relation between one-sided and two-sided equilibrium measures, we use corollary \ref{corollary2uniqueness} (stated originally for one-sided equilibrium measures) to prove the Bernoullicity of the two-sided equilibrium measure, using Ornstein theory. Then we turn to show that cohomology to a one-sided function can be done for two-sided potentials with summable variations, giving rise to a one-sided potential (which is also of summable variations). This is done, again, in a non-compact setting using the proof in \cite{coel}.
\section{Reduction to the topologically mixing case}\label{transection}
Suppose we know the following is true:
\begin{theorem}\label{reduced theorem}
Let $\hat \mu$ be an equilibrium measure of a Walters potential $\hat f\in C(\hat X)$ with finite first variation ($\vari_1 \hat f <\infty$) on a two-sided topologically mixing CMS. Assume $\sup \hat f <\infty$, $h_{\hat \mu}(\hat T)$ and $\int \hat f d\hat\mu< \infty $. Then $(\hat X,\hat {\mathcal{B}}, \hat \mu ,\hat T)$ is measure theoretically isomorphic to a Bernoulli scheme.
\end{theorem}
We can use the following lemma in order to show theorem \ref{reduced theorem} implies theorem \ref{main theorem}.
\begin{lemma}[\cite{ASS}]\label{arse}
Let $(X,\mathcal{B},\mu, T)$ be an ergodic invertible probability preserving transformation with a measurable set $X_0$ s.t.
\begin{enumerate}
\item $T^p(X_0) = X_0 \mod \mu$,
\item $X_0,T(X_0),...,T^{p-1}(X_0)$ are pairwise disjoint $\mod \mu$,
\item $T^p: X_0 \to X_0$ equipped with $\mu(\cdot|X_0)$ is a Bernoulli automorphism
\end{enumerate}
Then $(X,\mathcal{B},\mu, T)$ is measure theoretically isomorphic to the product of a Bernoulli scheme and a finite rotation. 
\end{lemma}
Let $\hat X, \hat \phi$ etc. be as in theorem \ref{main theorem}. By the spectral decomposition (Remark 7.1.35 in \cite{kitchen}), there exist $\hat X_0,\hat X_1,...\hat X_{p-1},p\in \bb{N}$ s.t $X_i$ are pairwise disjoint $\mod \hat \mu, \hat T(\hat X_i) = \hat X_{i+1 \mod p}$ and $(\hat X_i, \hat T^p)$ is topologically mixing. Since we assume we know theorem \ref{reduced theorem} to be true, this implies $\hat T^p$ is Bernoulli. It is known that $\hat \mu$, as an equilibrium measure, is ergodic \cite{sar} theorem 4.7, so the hypotheses in lemma \ref{arse} are satisfied and \ref{main theorem} holds. Thus, wlog, we may restrict ourselves to topologically mixing TMS and prove (under the conditions of theorem \ref{reduced theorem}) they are measure theoretically isomorphic to Bernoulli schemes. 
\section{Cohomology to one-sided function - Walters case}\label{waltersection}
Let $(Y,S)$ be a dynamical system on $Y$, a  metric space. We define \textit{Bowen's metric at time $n$} as follows. $d_{n}(x,y) :=\max _{0\leq k<n}d(T^{k}x,T^{k}y)$. Now we let $g: Y\to \bb{R}$. We say $g$ is \textit{Walters} (satisfies Walters condition, has the Walters property) \cite{walterscondition} if $\forall \epsilon >0,\ \exists \delta>0, \ s.t. \ \forall n\geq 1, \ \forall x,y \in Y: \ d_{n}(x,y) < \delta \Rightarrow |f_{n}(x) -f_{n}(y)| <\epsilon$. The careful reader may check that this definition specializes to the definitions we'll present (and use) for the special case of TMS. Details can be found in e.g. \cite{bou}.

We use hats (e.g. $\hat x\in \hat X$, $\hat T$ etc.) in order to distinguish objects defined using the two-sided shift space from ones defined on the one-sided shift space. When no confusion may arise, we might drop the hats. Let $f:X \to \bb{R}$. Its $n$th \textit{variation} is defined as $\vari_{n} f:= \sup\{| f(x)-f(y)|:x_i = y_i\ \forall  0\leq i \leq n-1 \}$. Now let $\hat f:\hat X \to \bb{R}$. Its $[-k,n+k]$ \textit{asymmetric variation} is defined as $\vari_{[-k,n+k]} \hat f:= \sup\{|\hat f(\hat x)-\hat f(\hat y)|:x_i = y_i\ \forall  -k\leq i \leq n+k-1 \}$. Denote the $n$th Birkhoff sum of $f$ by $f_n = \sum_{i=0}^{n-1} f\circ T^i$. This is how we the above definition specializes for CMSs:
\begin{definition}\label{walters1definition}
Let $(X,T)$ be a one-sided CMS, $f\in C(X)$. $f$ is said to satisfy Walters condition if $\lim\limits_{k\to \infty}\sup_{n\geq 1} \vari_{n+k} f_n = 0$ and $\forall k\geq 1, n\geq 1, \vari_{n+k} f_n <\infty$.\\
\end{definition}
\begin{definition}\label{walters2definition}
Let $(\hat X,\hat T)$ be a two-sided CMS, $\hat f\in C(\hat X)$. $\hat f$ is said to satisfy Walters condition if $\lim\limits_{k\to \infty}\sup_{n\geq 1} \vari_{[-k,n+k]} \hat f_n =0$ and $\forall k\geq 1, n\geq 1, \vari_{[-k,n+k]} \hat f_n <\infty$.
\end{definition}
If $f$ is Walters, then it is uniformly continuous. However, it need not be bounded.
\begin{theorem}\label{waltercohomology}
Let $(\hat X,\hat T)$ be a two-sided TMS (possibly with countable alphabet). Let $\hat f$ be Walters with $\vari_{1} <\infty$. Then there exists a one-sided $f:X\to \bb{R}$ that is also Walters s.t $\hat f +h- h\circ T= f \circ \pi$ where $h$ is bounded and uniformly continuous.
\end{theorem}
\begin{proof}
Following Sinai \cite{sinai}, for every $a \in S$ we define $z^a$ - some arbitrary left infinite sequence that can precede $a$. Let $\hat x\in \hat X$. Define $\bar x$ to satisfy $(\bar x_{i})_{-\infty}^{-1} = z^{\hat x_{0}}$, $ (\hat x_{i})_{0}^{\infty} = (\bar x_{i})_{0}^{\infty}$. Let
$$
h(\hat x) := \sums.
$$
We claim that $h$ is well defined, uniformly continuous and bounded. To see this, note that $\hat T$ is uniformly continuous - for $\epsilon > 0$ choose $\delta = \eps / 2$. $\hat x \mapsto \bar x$ is just a projection and also uniformly continuous.

Denote $H_k(\hat x) := \sum_{i=0}^{k-1} \hat f(\hat T^i\hat x) -\hat f (\hat T^i\bar x)$. $H_k$ is uniformly continuous as a sum and compositions of such. Now we show the series  $H_{k}$ is uniformly Cauchy. Let $\epsilon > 0$. We want to find $N$ so that $n_1 \geq N$ and $k > 0$ imply  $|\sum_{i=n_1}^{n_1+k-1} [\hat f(\hat T^i\hat x) - \hat f(\hat T^i\bar x)]| <\eps$. Since $\hat f \in W_{}(\hat X,\hat T)$, there exists $M\geq 1$ such that if $(\hat x_{i})_{-M}^{M+n} = (\hat y_{i})_{-M}^{M+n}$ then  $|\sum_{i=0}^{n-1} \hat f(\hat T^i\hat x) -\hat f(\hat T^i\hat y)|<\epsilon$. Let $N=M+1$. For $n_1 \geq N$,  $(\hat T^{n_1}\hat x)_{-M}^{\infty} = (\hat T^{n_1}\bar x)_{-M}^{\infty}$ and we get  $|\sum_{i=0}^{k-1} \hat f(\hat T^i(\hat T^{n_1}\hat x)) -\hat f(\hat T^i(\hat T^{n_1}\bar x))|<\epsilon$ for all $k\geq 1$. Some rephrasing gives $|\sum_{i=n_1}^{n_1+k-1} [\hat f(\hat T^i\hat x) - \hat f(\hat T^i\bar x)| <\epsilon$ for all $k\geq 1$, showing the uniform Cauchy property of $\{H_k\}_{k=1}^{\infty}$. This shows the uniform continuity of $h$ (clearly it is well defined).

By the uniform Cauchy property of $\{H_k\}_{k=1}^{\infty}$,  $H_{k} \rightarrow h$ uniformly and there's some $k$ for which $|h(\hat x) - H_k(\hat x)| <1, \forall \hat x \in \hat X$. For that $k$, we see that  $H_k(\hat x) \leq k\cdot\vari_{1}f <\infty $. This shows $h$ is bounded.

Now we turn to construct the appropriate transfer function and prove the cohomology. Let $\hat x \in \hat X$. Then,
\begin{eqnarray*}
\hat f (\hat x) - h(\hat x) + h(\hat T\hat x)&=& \hat f(\hat x) - \sum_{i=0}^{\infty} \hat f(\hat T^i\hat x) -\hat f(\hat T^i\bar x) + \sum_{i=0}^{\infty} \hat f(\hat T^i[\hat T\hat x]) -\hat f(\hat T^i\overline {[\hat T\hat x]})\\
&=& \hat f(\bar x)- \sum_{i=1}^{\infty} \hat f(\hat T^i\hat x) -\hat f(\hat T^i\bar x) + \sum_{i=0}^{\infty} \hat f(\hat T^i[\hat T\hat x]) -\hat f(\hat T^i\overline{[\hat T\hat x]})\\
&=& \hat f(\bar x)- \sum_{i=0}^{\infty} \hat f(\hat T^i[\hat T\hat x]) -\hat f(\hat T^i[\hat T\bar x]) + \sum_{i=0}^{\infty} \hat f(\hat T^i[\hat T\hat x]) -\hat f(\hat T^i\overline{[\hat T\hat x]})\\
&=& \hat f(\bar x) -  \sum_{i=0}^{\infty} \hat f(\hat T^i\overline{[\hat T\hat x]}) -\hat f(\hat T^i[\hat T\bar x])\\
\end{eqnarray*}
Since the bottom expression depends only on positive coordinates (the appended $z^a$ were completely arbitrary - they were just required to let $\bar x$ be admissible) we get:
$$
\hat f (\hat x) -h(\hat x) + h(\hat T\hat x) = (f\circ \pi)(\hat x),
$$
for some $f: X \to \bb{R}$ (one-sided). Now we turn to show $f$ is Walters.
$$
\sup_{n\geq 1}[\vari_{[-k,n+k]} (\hat f - h + h\circ T)_n] \leq \sup_{n\geq 1}[\vari_{[-k,n+k]}\hat f_n] + \sup_{n\geq 1}[\vari_{[-k,n+k]} (h - h\circ T)_n].
$$
The first summand approaches zero by the Walters property of $\hat f$. As for the second, note that $(h - h\circ T)_n =h - h\circ T^n$.
\begin{eqnarray*}
\sup_{n\geq 1}[\vari_{[-k,n+k]} (h - h\circ T^n)]
&\leq& \sup_{n\geq 1}[\vari_{[-k,n+k]} h] + \sup_{n\geq 1}[\vari_{[-k,n+k]} h\circ T^n]\\
&\leq& 2\cdot \vari_kh \longrightarrow 0
\end{eqnarray*}
by uniform continuity of $h$. Since $h$ is bounded, $\vari_{n+k}f_n < \infty$ for all $k\geq 1$. So $f$ is Walters.
\end{proof}
The assumption that $\vari_1 \hat f <\infty$ is not too restricting since the Walters property implies $\vari_2 \hat f<\infty$ and we can recode the shift space using 2-cylinders. Then we see that in the new space $\hat f$ has finite first variation.

\section{The GRPF theorem with some consequences}\label{GRPF}
We will rely on the Generalized Ruelle-Perron-Frobenius theorem. This was originally proved in \cite{GRPF} for weakly H\"older potentials. The generalization to Walters potentials which we use may be found in \cite{sar}. Note that in this section we are solely concerned with one-sided TMS.
\begin{definition}[Non-singular maps] A measurable map $T$ on a measure space $(X,\cal{B},\mu)$ is called {\em non-singular} if $\mu \circ T_{-1}(E) = 0 \leftrightarrow \mu (E) =0$.
\end{definition}
We will also call a measure non-singular and mean the same.
\begin{definition}
Suppose $\mu$ is a non - singular measure on $X$ with set of states $S$. We let $\mu \circ T$ denote the measure on $X$ given by $(\mu \circ T)(E) := \sum_{a\in S} \mu[T(E \cap [a])]$.
\end{definition}
Note that generally $\mu(TE) \not = (\mu \circ T)(E)$. The following are standard facts:
\begin{fact}\label{nu comp T}
Let  $\nu$ be a non-singular measure on a TMS $X$. Then for all non-negative Borel functions $f:X\to\bb{R}$,
$$
\int_X f d\nu \circ T = \sum_{a\in S} \int_{T[a]}f(ax) d\nu(x)
$$
\end{fact}
\begin{proof}
We show for indicator functions. The same holds for any integrable Borel function.
\begin{eqnarray*}
\int 1_E d\nu \circ T &=& \sum_{a\in S} \int 1_{T(E\cap [a])} d\nu =\sum_{a\in S} \int 1_{\{x: T^{-1}x \in E\cap [a]\}} d\nu (x)\\
&=& \sum_{a\in S} \int 1_{\{T^{-1}x \in E\}} 1_{\{T^{-1}x \in [a]\}} d\nu (x)= \sum_{a\in S} \int _{T[a]}1_{\{T^{-1}x \in E\}} d\nu (x)\\
&=& \sum_{a\in S} \int_{T[a]}1_E(ax) d\nu (x)\\
\end{eqnarray*}
\end{proof}
\begin{fact}
Let $\nu$ be a non-singular measure on a TMS. Then $\nu \ll \nu \circ T$.
\end{fact}
\begin{definition}[Jacobian] Let $\mu$ be a non-singular Borel measure on $X$. $g_{\mu}:= \frac{d\mu}{d\mu \circ T}$ is called $\mu$'s {\em Jacobian}. If $\mu \sim \mu \circ T$, then $\log g_{\mu}$ is called the {\em log-Jacobian}.
\end{definition}
\begin{definition}[The transfer operator]. The {\em transfer operator} of a non-singular measurable map on a $\sigma$-finite measure space $(X,\mathcal{B},\mu)$ is the operator $\widehat{T_{\mu}}: L^1(X,\mathcal{B},\mu) \to L^1(X,\mathcal{B},\mu)$ defined by:
$$
\widehat{T_{\mu}}f:= \frac{d\mu_f \circ T^{-1}}{d\mu}, \text{ where } d\mu_f:= fd\mu.
$$
\end{definition}
One can check the transfer operator is well defined as a Radon-Nikodym derivative.
\begin{fact}[Formula for the transfer operator] Suppose $X$ is a TMS and $\mu$ is $T$ non-singular. Then the transfer operator of $\mu$ is given by
$$
(\widehat{T_{\mu}}f)(x) = \sum_{Ty=x} \frac{d\mu}{d\mu \circ T}(y)f(y)
$$
\end{fact}
\begin{fact}[Properties of the transfer operator]\label{ruelle properties} Let $\mu$ be a non-singular $\sigma$-finite measure on $X$. Then:
\begin{enumerate}
\item If $f \in L^1$, then $\widehat{T_{\mu}} f$ is the unique $L^1$-element s.t. for every $\vp \in L^{\infty}$,
$$
\int \vp \widehat{T_{\mu}} f d\mu = \int \vp \circ T f d\mu
$$
This means that the transfer operator behaves like the adjoint of the Koopman operator (except that it acts on $L^1$, not on $L^2$).
\item $\widehat{T_{\mu}}$ is a bounded linear operator on $L^1$, and $||\widehat{T_{\mu}}|| = 1$.
\item $\widehat{T_{\mu}}^* \mu = \mu$.
\end{enumerate}
\end{fact}
\begin{definition}[Ruelle Operator] Suppose $X$ is a TMS, $\phi : X \to \mathbb{R}$. The {\em Ruelle operator} of $\phi$ is $\Lp f = \sum_{Ty=x} e^{\phi(y)} f(y)$.
\end{definition}
The definition is not proper since we did not state what is the domain and range. In our case the sum might even be infinite, since we may have infinitely many preimages for every point (recall we wish to consider infinite state TMS). However, we will restrict ourselves to functions the satisfy Walters condition and for such functions this operator turns out to be well defined and well behaved. Note that in the particular case where $\phi$ is the log jacobian of some measue $\mu$, then $\Lp = \widehat{T_{\mu}}$ and all the good properties of the transfer operator hold also for the Ruelle operator. 
\begin{definition}[Conservative measures] A non-singular map $T$ on a $\sigma$-finite measure space $(\Omega, \cal{B},\nu)$ is called {\em conservative} if every set $W \in \cal{B}$ s.t. $\{T^{-n}W\}_{n\geq 0}$ are pairwise disjoint $\mod \nu$ satisfies $W = \varnothing \mod \nu$.
\end{definition}
Recall that a TMS is topologically mixing if for every two states $a,b$ there exists $N_{ab}$ s.t. $\forall n \geq N_{ab}$ there exists some $\xi_i, 1\leq i \leq n-1$ s.t.  $A_{a\xi_1}A_{\xi_1\xi_2}\dots A_{\xi_{{(n-1)}b}} = 1$.
\begin{definition}[Preliminary combinatoric expressions] Let $\phi: X \to \bb{R}$.
\begin{itemize}
\item $\vp_a(x):= 1_{[a]} \min\{n\geq 1|T^n(x) \in [a]\}$ - the first return time.
\item $Z_n^*(\phi,a):= \sum_{T^nx=x} e^{\phi_n(x)}1_{[\vp_a=n]}(x)$.
\item $Z_n(\phi,a):= \sum_{T^nx=x} e^{\phi_n(x)}1_{[a]}(x)$.
\end{itemize}
\end{definition}
\begin{proposition}[Gurevich pressure]
\ass For every state $a\in S$, $\lim\limits_{n\to \infty} \frac{1}{n} \log Z_n(\phi,a)$ exists and is independent of $a$. We call this limit the {\em Gurevich pressure} of $\phi$ and denote it $\pg$.
\end{proposition}
\begin{definition}[Modes of recurrence] Suppose $X$ is a topologically mixing TMS, $\phi: X\to \bb{R}$ is Walters and $\pg < \infty$. Let $\lambda := \exp (\pg)$. Fix some state $a$. Then
\begin{itemize}
\item $\phi$ is called {\em recurrent}, if $\sum \lambda^{-n}Z_n(\phi,a) =\infty$ and {\em transient} otherwise;
\item $\phi$ is called {\em positive recurrent}, if it is recurrent and $\sum n\lambda^{-n}Z^*_n(\phi,a) <\infty$;
\item $\phi$ is called {\em null recurrent}, if it is recurrent and $\sum n\lambda^{-n}Z^*_n(\phi,a)=\infty$.
\end{itemize}
\end{definition}
\begin{theorem}[Generalized Ruelle's Perron-Frobenius, \cite{GRPF}]\label{GRPFtheorem} \ass Assume $\sup \pg < \infty$.
\begin{enumerate}
\item $\phi$ is positive recurrent if there exists $\lambda >0$, $h$ positive continuous and $\nu$ conservative finite on cylinders s.t. $\Lp h = \lambda h, \Lp^*\nu = \lambda \nu$ and $\int hd\nu <\infty$. In this case, $\lambda = \exp(\pg)$. Also, for every cylinder $[\underline a]$, $\lambda^{-n}\Lp^n1_{[\underline a]} \rightarrow\frac{h\nu [\underline a]}{\int h d\nu}$ uniformly on compacts.

\item $\phi$ is null recurrent if there exists $\lambda >0$, $h$ positive continuous and $\nu$ conservative finite on cylinders s.t. $\Lp h = \lambda h, \Lp^*\nu = \lambda \nu$ and $\int hd\nu = \infty$. In this case, $\lambda = \exp(\pg)$. Also, for every cylinder $[\underline a]$, $\lambda^{-n}\Lp^n1_{[\underline a]} \rightarrow 0$ uniformly on compacts.

\item $\phi$ is transient if there is no $\nu$ conservative finite on cylinders $\Lp^* \nu = \lambda \nu$ for some $\lambda >0$.
\end{enumerate}
\end{theorem}
If $X$ is compact, then $\phi$ is positive recurrent \cite{sar}, so the last two cases cannot occur. Here we are considering non-compact shift spaces and the last two cases may, in fact, occur. We call the probability  measure $dm:= h d\nu$ (apply normalization if required) from part (1) a \textit{RPF measure}. Our  focus will be on the recurrent case. We turn to some consequences of the GRPF theorem.
\begin{proposition}[Uniqueness of RPF measures]\label{UniquenessRPF} \ass Then $\phi$ has at most one RPF measure.
\end{proposition}
\begin{proof}
Let $\mu, \nu$ be two RPF measures with corresponding $h,f$. Their eigenvalue is equal: $\lambda = e^{\pg}$. Now, for any $[\underline a]$ we have $\frac{h\mu [\underline a]}{\int h d\mu} = \frac{f\nu[\underline a]}{\int fd\nu} \Rightarrow h = f\cdot const$, so the eigenspace of $\lambda$ is 1-dimensional. Uniqueness of $h$ up to multiplication follows.

Let $f=ch$. We get $\frac{h\mu [\underline a]}{\int h d\mu} = \frac{ch\nu[\underline a]}{\int chd\nu}= \frac{h\nu[\underline a]}{\int hd\nu}$. $h>0$ so we may divide :$\frac{\mu [\underline a]}{\int h d\mu} = \frac{\nu[\underline a]}{\int hd\nu}$. Thus $\mu[\underline a] = \nu[\underline a] \cdot \frac{\int hd\mu}{\int hd\nu}$.

This can be extended to the algebra generated by cylinders. Caratheodory's extension theorem extends this to the Borel $\sigma$-algebra and completes it. Thus $\mu = \nu$ up to a multiplicative factor. We require $\int h d\nu =1$, so the RPF measure is indeed uniquely determined.
\end{proof}
\begin{definition}[Equilibrium measure] Let $X$ be a TMS, $\phi: X\to \bb{R}$ measurable. A shift invariant probability measure $m$ is called an equilibrium
measure for $\phi$ if
$$
h_m(T)+ \int \phi dm = \sup\{h_{\mu}(T)+ \int \phi d\mu\}
$$
where the supremum ranges over all invariant Borel probability measures $\mu$ for which $h_{\mu}(T)+ \int \phi d\mu$ is well defined (i.e. does not equal $\infty - \infty$).
\end{definition}
The following variational principle was proved in a compact setting by Ruelle \cite{ruellevariational} (see also \cite{wal2}). Sarig \cite{sarig1999thermodynamic} showed this for countable (i.e. non compact) Markov shifts. 
\begin{theorem}[Variational principle]\label{variational principle} 
\ass  If $\sup \phi < \infty$ then
$$
\pg = \sup \{h_{\mu}(T) +\int \phi d\mu\},
$$
where the supremum ranges over shift invariant Borel probability measures for which $h_{\mu}(T) +\int \phi d\mu$ is well defined.
\end{theorem}
The proof here follows \cite{sar} almost verbatim. We give a proof here since there it is not stated for functions that are Walters, but summable variations. Before we prove this, we state few useful facts. We start with a lemma from \cite{sar}. The proof there is stated for potentials with summable variations but the same proof works verbatim if the potential is Walters. For that lemma we need the following definition.
\begin{definition}[Sub-system]
Let $X$ be a TMS over the set of states $S$ and with transition matrix $\bb{A}= (t_{ij})_{S\times S}$. A sub-system of $X$ is a TMS with a set of states $S' \subseteq S$ and transition matrix $\bb{A} '= (t'_{ij})_{S'\times S'}$ s.t. $t'_{ij}=1 \Leftrightarrow t_{ij} =1$.
\end{definition}
\begin{lemma}[Pressure over sub - systems]\label{pressure over sub systems}
\ass Then $\pg = \sup_Y\{\pg\}$, where the supremum ranges over $Y$'s that are topologically mixing compact sub-systems of $X$.
\end{lemma}
In this context it might be useful for some readers to recall that a TMS is compact iff it has a finite number of states.
We will also need the following.
\begin{definition}[Sweep-out set]
$E$ is a sweep out set for a probability preserving transformation $(\Omega,\mathcal{B},\mu,T)$ if $T^n(x) \in E$ for some $n\geq 1$ for $\mu$ a.e. $x\in\Omega$.
\end{definition}
Let $E$ be sweep out. Then $\mu (E)> 0$. If $\mu$ is ergodic, any set of positive measure is a sweep out set. The following are classic results.
\begin{fact}[Abramov's Formula]
Let $E$ be a sweep-out set. If $\mu_E := \mu (\ \cdot\ |E)$ then
$$
\ent{\mu} = \mu (E) h_{\mu_{E}}(T_{E})
$$
\end{fact}
\begin{fact}[Kac's formula]
Let $E$ be a sweep-out set. Then for every $f\in L^1(\mu)$,
$$
\int f d\mu =\int_E \sum_{k=0}^{\vp_E-1} f\circ T^k d\mu
$$
Particularly, $\int_E \vp_E d\mu_E = 1/ \mu (E)$.
\end{fact}
\begin{fact}[Rokhlin's Formula]
Let $X$ be a TMS on alphabet $S$. Let $\alpha:= \{[a]| a\in S\}$. Then:
\begin{itemize}
\item if $H_{\mu}(\alpha) < \infty$ then $\ent{\mu} = - \int \log \frac{d\mu}{d\mu \circ T} d\mu$.
\item if $H_{\mu}(\alpha)=\infty$ then $\ent{\mu} \geq - \int \log \frac{d\mu}{d\mu \circ T} d\mu$.
\end{itemize}
\end{fact}
\begin{proof}[Proof of the variational principle, theorem \ref{variational principle}]
The proof follows \cite{sarig1999thermodynamic}. Let $\mu$ be $T$-invariant. We first show that $\pg \geq \sup \{h_{\mu}(T) +\int \phi d\mu\}$. If $\pg = \infty$ this is trivial. Assume $\pg < \infty$. Now, if $\int \phi d\mu = - \infty$ then $\ent{\mu} < \infty$ (so that their difference is well defined). Thus $ \pg \geq \{\ent{\mu} + \int \phi d\mu\} = - \infty$ so this case is also trivial. Now we assume $\int \phi d\mu > -\infty$.\\
wlog we assume $S = \bb{N}$. Set $\alpha_m : = \{[1],[2],...,[m], [>m]\}$ where $[>m]:= \cup_{n > m}[n]$. Let $\mathcal{B}_m := \sigma(\alpha_m)$ (the minimal $\sigma$-algebra that contains $\bigvee_{i=0}^{p-1} T^{-i}\alpha_m$ for every $p\geq 1$). As $m \to \infty$, $\mathcal{B}_m \uparrow \mathcal{B}$ and so, 
\begin{eqnarray*}
\lim_{m \rightarrow \infty} h_{\mu}(T,\alpha_m) + \int \phi d \mu =  h_{\mu}(T) + \int \phi d \mu.
\end{eqnarray*}
Fix $m$ and set $\beta:= \alpha_m$. Let $a_i , 0 \leq i < n$ be atoms of the partition $\beta$ and $\underline a := (a_0,...,a_{n-1})$ an $n$-tuple of those atoms. Define
$$
\la \underline a \ra = \la a_0,...,a_{n-1}\ra := \bigcap_{k=0}^{n-1}T^{-k}a_k,
$$
and set $\phi_n \langle \ua \rangle := \sup\{\phi_n(x): x\in \la \underline a \ra\}$. Recall $\mu$ is $T$-invariant so
\begin{eqnarray*}
\frac{1}{n}H_{\mu}(\beta_{0}^{n}) + \int \phi d\mu &=& \frac{1}{n}\left (H_{\mu}(\beta_{0}^{n}) + \int \phi_n d\mu \right )\leq
\frac{1}{n} \sum_{\lar\in \beta_{0}^{n}} \mu \lar \log \frac{e^{\phi_n \lar}}{\mu \lar}\\
&=& \frac{1}{n} \sum_{a,b \in \beta} \mu(a\cap T^{-n}b) \sum_{\lar \subseteq a\cap T^{-n}b, \lar \in \beta_0^n} \mu (\lar | a\cap T^{-n}b) \log
\frac{e^{\phi_n\lar}}{\mu \lar}\\
&\leq & \frac{1}{n} \sum_{a,b \in \beta} \mu (a \cap T^{-n}b) \log \sum_{\lar \subseteq a\cap T^{-n}b, \lar \in \beta_0^n} e^{\phi_n \lar} +
\frac{1}{n} H_{\mu}(\beta \vee T^{-n} \beta)= (*)\\
\end{eqnarray*}
To see the last inequality, first use Jensen's inequality, then Bayes' rule to get $\frac{\mu(\lar | a\cap T^{-n}b)}{\mu \lar} =\frac{1}{\mu (a \cap T^{-n}b)}$ and simplify. We have
$$
(*) =: \sum_{a,b \in \beta} \mu (a\cap T^{-n}b)P_n(a,b) + O(\tfrac{2}{n} H_{\mu}(\beta)),
$$
where
$$
P_n(a,b) := \frac{1}{n} \log \sum_{\lar \subseteq a\cap T^{-n}b, \lar \in \beta_0^n} e^{\phi_n\lar}.
$$
Taking limit as $n \to \infty$, recalling $\beta = \alpha_m$ we obtain
$$
h_{\mu}(T,\alpha_m) + \int \phi d\mu \leq \limsup_{n\to \infty} \left \{ \sum_{a,b \in \beta} \mu (a\cap T^{-n}b)P_n(a,b) \right \}.
$$
We need to estimate $P_n(a,b)$.

First, assume $a,b \not = [>m]$. Let $M := \sup_{n,k\geq 1}\vari_{n+k}\phi_n$. This exists by the Walters property of $\phi$. Define further for every finite word $\underline a$ its periodic orbit $x_{\underline a}:= (\underline a,\underline a,...)$. Denote the natural partition $\alpha:= \{[1],[2],[3],...\}$ (compare with $\beta = \alpha_m = \{[1],[2],[3],...,[m],[>m]\}$). Since $\alpha_0^n$ is finer than $\beta_0^n$,
\begin{eqnarray*}
P_n(a,b) &=& \frac{1}{n} \log \sum_{\lar \subseteq a\cap T^{-n}b, \lar \in \beta_0^n} e^{\phi_n\lar}\\
&\leq & \frac{1}{n} \log \sum_{[\underline a] \in \alpha_{0}^{n}, [\underline a] \subseteq [a]} e^{\phi_n(x_{\underline a})+M}\\
&=&\frac{M}{n} + \frac{1}{n} \log\sum_{T^{n}x=x} e^{\phi_n(x)}1_{[a]}(x) \xrightarrow[n\to\infty]{} \pg.\\
\end{eqnarray*}
so,
$$
a, b \not = [>m] \Longrightarrow \limsup_{n\to \infty} P_n(a,b) \leq \pg.
$$
Next, we assume $a= [>m]$ or $b= [>m]$. For every $\lar \in \beta_0^n$ s.t. $\lar \subseteq a\cap T^{-n}b$,
$$
\lar = \langle >m,...,>m,\xi_1,...,\xi_j,>m,...,>m \rangle,
$$
where either $i \geq 1$ or $k \geq 1$, $\xi_1, \xi_j\not = [>m]$ and $i+j+k = n+1$. Recalling that $\sup \phi < \infty$ by assumption, for such $i,j,k$ -  $\phi_n \lar \leq (n+1-j) \sup \phi + \phi_j\langle \xi_1,...,\xi_j\rangle$. Summing over all possibilities we have
\begin{eqnarray*}
P_n(a,b) &\leq & \frac{1}{n} \log \left (\sum_{i=0}^{n+1}\sum_{j=1}^{n+1-i}\sum_{\xi_1 =1}^{m}\sum_{\xi_j=1}^{m} e^{jP_j(\xi_1,\xi_J) + (n+1 -j)\sup \phi}   \right )\\
&\leq & \frac{1}{n} \log \left (   e^{\phi_{n+1}\langle >m,...,>m \rangle} +n \sum_{\xi,\eta \not = [>m]} \sum_{j=1}^n
e^{jP_j(\xi,\eta) + n\sup \phi} \right )\\
& \leq & \sup \phi + \frac{1}{n} \log \left (1+ n \sum_{\xi,\eta \not = [>m]} \sum_{j=1}^n e^{jP_j(\xi,\eta)}  \right )\\
&\leq & \sup \phi + \frac{1}{n} \log \left (2\cdot n \sum_{\xi,\eta \not = [>m]} \sum_{j=1}^n e^{jP_j(\xi,\eta)}  \right ) \text{ (for $n\geq N_0$ large enough) }\\
& =& C+ \frac{1}{n}\log \left (\sum_{\xi,\eta \not = [>m]} \sum_{j=1}^n e^{jP_j(\xi,\eta)}  \right ) \text{ ($C$ is just some constant)}\\
\end{eqnarray*}
Since for $\xi,\eta \ne [>m]$ we know that $\limsup\limits_{j\to\infty}P_j(\xi,\eta) \leq \pg$. We conclude that
$$
a = [>m]  \text{ or } b =[>m] \Longrightarrow \limsup_{n\to \infty}P_n(a,b) \leq C',
$$
where $C'$ is some constant.\\
We can now turn to finish the proof using the above analysis.
\begin{eqnarray*}
h_{\mu}(T, \alpha_m) +\int \phi d\mu &\leq & \limsup_{n \to \infty} \lbrace \sum_{a,b \in \beta} \mu(a\cap T^{-n}b)P_n(a,b) \rbrace \\
&\leq & \limsup_{n \to \infty} \lbrace \pg \sum_{a,b \not = [>m]} \mu (a\cap T^{-n}b) + C \sum_{\neg(a,b \not = [>m])} \mu (a\cap T^{-n}b) \rbrace \\
&\leq & \pg \mu (X) + C\mu [>m] + C\mu (T^{-n}[>m])\\
&=& \pg + o(1), \text{ as $m \to \infty$}.
\end{eqnarray*}
We now show the inverse inequality. Fix some $\epsilon >0$ and a topologically mixing compact sub-system $Y\subseteq X$ s.t. $\pg \leq P_G(\phi|_{Y}) +\epsilon$, by the previous lemma on pressure over sub-systems (lemma \ref{pressure over sub systems}). Denote $\psi := \phi|_{Y}$. Since $Y$ is compact, the GRPF theorem (actually, the original Ruelle's Perron-Frobenius theorem suffices here), $\psi$ is positive recurrent and so there exists a positive eigenfunction $h>0$ for Ruelle's operator and a conservative \textit{probability} measure $\nu$ on $Y$ such that $L_{\psi}h = e^{P_G(\psi)}h, L^*_{\psi}\nu =e^{P_G(\psi)}\nu, \int hd\nu = 1$. This measure is indeed a probability measure since it is finite on cylinders and we only have finitely many cylinders of any fixed length. We set $dm = hd\nu$. This is a shift invariant probability measure, since
$$
m(T^{-1}[\underline a]) = \nu(h1_{[\underline a]}\circ T) = \nu( e^{-P_G(\psi)}(L_{\psi}h1_{[\underline a]})) = m [\underline a].
$$
One consequence of the above is that $\psi = \log \frac{d\nu}{d\nu \circ T} +P_G(\psi)$, by the properties of the transfer operator (fact \ref{ruelle properties}) that also apply to Ruelle's operator.
Now, let $\alpha_Y :=\{[a]\cap Y | a\in S'\}$, where we let $S'$ denote the alphabet over which $Y$ is defined. Since $Y$ is compact, $\alpha_Y$ is finite and so $H_m(\alpha_Y) <\infty$. This means we may use Rokhlin's formula:
\begin{eqnarray*}
h_m(T|_Y) &=& -\int_Y \log \frac{dm}{dm\circ T} dm = - \int_Y \log \frac{h}{h\circ T}\frac{d\nu}{d\nu \circ T}dm\\
&=& -\int_Y [\psi -P_G(\psi) +\log h - \log h \circ T] dm\\
&=&P_G(\psi) - \int_Y \psi dm \text { (justification below) }.\\
\end{eqnarray*}
The last equality holds, since $\log h$ is continuous on a compact space ($h>0$) and hence absolutely integrable and $m$ is $T$-invariant.
Thus, $h_m(T|_Y) + \int_Y \psi = P_G(\psi)$. This implies that $P_G(\psi) \leq \sup\{h_{\mu}(T) + \int \phi d\mu\}$. Since by construction, $\pg \leq P_G(\psi) +\epsilon$, we get that $\pg \leq \sup\{h_{\mu}(T) + \int \phi d\mu\}+ \epsilon$. But $\epsilon$ was arbitrary, so we're done.
\end{proof}
\begin{definition}[$g$-functions, \cite{keane}] 
Let $X$ be a TMS. $g:X\to (0,1]$ is a (sub) g-function if $\sum_{Ty=x} g(y) = 1 (\leq 1)$.
\end{definition}
\begin{theorem}[Cohomology for $g$-functions]\label{cohomology g functions}
\ass Suppose $\pg < \infty$.
\begin{enumerate}
\item  If $\phi$ is recurrent, then $\phi - \pg = \log g+ \vp - \vp \circ T$ where $g$ is a g-function, $\log g$ is Walters and $\vp$ continuous .
\item  If $\phi$ is transient, then $\phi - \pg = \log g+ \vp - \vp \circ T$, where $\log g$ is Walters, $g$ is a sub g-function and $\vp$ continuous.
\end{enumerate}
In both cases the cohomology can be done s.t. $\vari_{1}\vp < \infty$.
\end{theorem}
Proof may be found in \cite{sar,sarig2001phase}. This will be used in the reduction at the beginning of the proof of the uniqueness theorem (theorem \ref{uniqueness theorem}).
%
\section{The uniqueness theorem}\label{uniqueness}
The focus of this section is on proving the following theorem, which is a generalization of a theorem from \cite{BuzSar} that was proved there for summable variations potentials. For the uniqueness of equilibrium measures on compact spaces, see \cite{bowen1974unique}.
\begin{theorem}[Uniqueness of equilibrium measures]\label{uniqueness theorem}
\ass Let $\sup \phi < \infty, \pg<\infty$. Then
\begin{enumerate}
\item $\phi$ has at most one \eqm.
\item this \eqm, if exists, equals the RPF measure of $\phi$.
\item In particular, if $\phi$ has an \eqm then $\phi$ is positive recurrent and the RPF measure has finite entropy.
\end{enumerate}
\end{theorem}
\begin{proof}
Let us first assume $\mu$ is an \eqm for $\phi$. By subtracting a constant from $\phi$ we may assume, wlog, that $\pg = 0$ . By the previous theorem, $\phi = \log g+ \vp - \vp \circ T$ where $g$ is sub g-function, $\vp$ is continuous and $\vari_1\vp < \infty$.
We first show that
$$
\Lp e^{-\vp} = e^{-\vp}\ \mbox{and} \ \Lp^*(e^{\vp}\mu) =e^{\vp} \mu.
$$
This implies that $\mu= e^{-\vp}(e^{\vp}\mu)$ is an RPF measure. By proposition \ref{UniquenessRPF} it is unique. We divide the proof into several claims.
\begin{claim}
$\phi, \log g, \vp - \vp \circ T \in L^1(\mu)$ and  $\int (\vp -\vp \circ T) d\mu = 0$. Also, let $\mu = \int_{X} \mu_x d\mu(x)$, the ergodic decomposition of $\mu$, then $\int( \vp - \vp \circ T) d\mu_x = 0$ for $\mu$ a.e. $x$.
\end{claim} 
\begin{proof}
By assumption, $\sup \phi < \infty$. Note that $\int \phi d\mu > -\infty$, since $0 = \pg = \ent{\mu} + \int \phi d\mu$ is well defined ($\mu$ is an \eqm and entropy is non-negative). Hence $\phi \in L^1(\mu)$.\\
This implies $\phi \in L^{1}(\mu_{x})$ for $\mu$-a.e. $x$. Since $\log g \leq 0$ ($g$ is a sub g-function), $\vp - \vp \circ T = \phi - \log g$ is one sided integrable in $\mu_x$ for $\mu$ a.e. $x$ (i.e. $\int \vp - \vp \circ T d\mu_x$ is real or $\infty$ for $\mu$ a.e. $x$). We may use Birkhoff's theorem on the generic $x$'s from above.\\
This is how we do this. First, to make notation easy, let $\vp - \vp \circ T =: \psi$.  For a fixed $x$, if $\int \psi d\mu_x <\infty$, Birkhoff's theorem applies without further justification. If, however, $\int \psi d\mu_x = \infty$ (recall $\psi$ is one-sided integrable), let $M>0$. Then $\psi 1_{[\psi <M]}$ is integrable in $\mu_x$. Now $(\psi 1_{[\psi<M]})_n \leq \psi_n$. By Birkhoff's theorem,
$\frac{1}{n}\liminf \limits_{n\to\infty}\psi_n \geq \lim\limits_{n\to\infty}\frac{1}{n} (\psi 1_{[\psi <M]})_n =\int \psi 1_{[\psi <M]} d \mu_x$, $\mu_x$ a.e. (recall $\mu_x$ is ergodic). Now let $M \to \infty$ and get that $\lim\limits_{n\to\infty}\psi_n = \infty$, $\mu_x$ a.e. so Birkhoff's theorem still holds. This implies:
$$
\int(\vp -\vp \circ T)d\mu_{x} = \lim_{n\rightarrow \infty}\frac{1}{n} \sum_{k=0}^{n-1} (\vp -\vp \circ T)\circ T^{k}= \lim_{n\to \infty} \frac{1}{n}(\vp - \vp \circ T^{n})\text{   a.e. in $\mu_x$}
$$
By Poincar\'e's recurrence theorem, $|\vp - \vp \circ T^n| \leq 1$ i.o., so the limit equals zero.
This holds for a.e. ergodic component, so $\vp - \vp \circ T \in L^1(\mu)$ and $\int(\vp - \vp \circ T) d\mu= 0$. Also $\log g \in L^1(\mu)$. This completes the proof of the claim.
\end{proof}
Having proven the first claim, we proceed with another claim.
\begin{claim} $\mu$ is an at most countable convex combination of equilibrium measures $\mu_i$ s.t. for each $i$ there is a state $a_i$ s.t. $[a_i]$ is a sweep out set.
\end{claim}
\begin{proof}
Let $\int_X\mu_xd\mu(x)$ be the ergodic decomposition of $\mu$. Let $\{a_1,a_2,...\}$ be a list of states s.t. $\mu[a_i] > 0$. Let
$$
E_i:= \{x\in X|\mu_x[a_1],...,\mu_x[a_{i-1}] = 0, \mu_{x}[a_i] > 0\}.
$$
Note that $E_i$ are measurable (as intersection of such) and disjoint. Clearly $\biguplus E_i = X$. Let $p_i:= \mu (E_i)$ so $\sum p_i =1$. We discard all $E_i$ s.t. $p_i = 0$. Define $\mu_i:= \frac{1}{\mu (E_i)}\int_{E_i}\mu_xd\mu(x)$. Note-
$$
\mu = \sum \mu (E_i)\cdot \mu_i =\sum p_i \mu_i.
$$
By construction, all the ergodic components of $\mu_i$ charge $[a_i]$ so  it is a sweep out set w.r.t every ergodic component of $\mu_i$. So  $ \mu_i [a_i]>0$ and $[a_i]$ is a sweep out set for $\mu_i$.\\
Now, assume $\{p_i\}$ is a finite collection of measures. Then for every $i$ we get by the variational principle that-
$$
\ent{\mu_i}+ \int \phi d \mu_i \leq \pg = 0
$$
and by affinity of the entropy and linearity of the integral
$$
\sum p_i( \ent{\mu_i}+ \int \phi d \mu_i ) = \pg =0.
$$
Hence, $\ent{\mu_i}+ \int \phi d \mu_i = 0$ and so $\mu$ is a convex combination of the required equilibrium measures.\\
Now we assume $\{p_i\}$ is countable. For any $N$ write $q_{N+1} = \sum_{i>N} p_i$ and so $\mu^{*}_{N+1}:= \frac{1}{q_N}\sum_{i>N} p_i\mu_i$ is a probability measure. Apply the same argument on $\mu_1,...,\mu_N,\mu^*_{N+1}$ and send $N$ to infinity. This gives the required decomposition and the claim is proved.
\end{proof}
We proceed towards the proof of the main theorem with yet another claim.
\begin{claim}
Using the same notation, for all $i$, $\ent{\mu_i} = -\int \log \frac{d\mu_i}{d\mu_i \circ T} d\mu_i$.
\end{claim}
\begin{proof}
Fix $i$. Let $\vp_{a_i}(x):X\rightarrow \bb{N}$ be the first return time to $[a_i]$ and $\overline{T} :=T^{\vp_{a_i}}$ the induced transformation on $[a_i]$. Since $\mu_i$ is an \eqm we have that $\overline{T}$ preserves $\overline \mu_i := \mu_i(\ \cdot\  |[a_i])$ and admits the Markov partition $\beta := \{[a_i, \xi_1,...,\xi_{n-1},a_i]| n\geq 1 \ \wedge \ \xi_j \not = a_i, 1\leq j\leq n-1\} \backslash \{\varnothing\}$. This is a generator for $\overline T$. Assume for a moment that $H_{\overline \mu_i}(\beta) < \infty$. This assumption implies the claim as follows:
\begin{eqnarray*}
\frac{1}{\mu_i[a_i]}\ent{\mu_i} = h_{\overline{\mu_i}}(\overline{T})
&=& - \int\log \frac{d\overline{\mu_i}}{d\overline{\mu_i} \circ \overline{T}}d\overline{\mu_i} \text{ (Abramov's and Rokhlin's formul\ae )}\\
&=&-\frac{1}{\mu_i[a_i]} \int _{[a_i]} \log \frac{d\mu_i}{d\mu_i \circ T^{\vp_{a_i}}} d\mu_i\\
&=&-\frac{1}{\mu_i[a_i]} \int _{[a_i]} \sum_{j=0}^{\vp_{a_i}-1} \log \frac{d\mu_i}{d\mu_i \circ T} \circ T^{j} d\mu_i \text{ (see below)}\\
&=& -\frac{1}{\mu_i[a_i]}\int_X \log \frac{d\mu_i}{d\mu_i \circ T} d\mu_i,
\end{eqnarray*}
Where we've used Kac's formula and fact \ref{nu comp T}.
We now show that $\beta$ is a generator with finite entropy. Define a Bernoulli measure $\overline \mu_i^B(\cap_{j=1}^n \overline T^{-j}B_j):= \prod_{j=1}^n \overline\mu_i (B_j)$ where   $B_j \in \beta$ are partition sets. Note that for any $B\in \beta$ we have $\overline \mu_i^B(B) = \overline\mu_i (B)$. This is easily seen to be a Bernoulli measure and so $h_{\overline \mu_i^B}(\overline T) = H_{\overline \mu_i^B}(\beta)= H_{\overline \mu_i}(\beta)$. We proceed to show $h_{\overline \mu_i^B}(\overline T) < \infty$. Now define
\begin{eqnarray*}
\mu_i^B(E) &:=& \mu [a_i] \int_{[a_i]} \sum_{j=0}^{\vp_{a_i}-1} 1_E \circ T^j d\overline\mu_i^B . \\
\end{eqnarray*}
Note that by its definition with the aid of Kac's formula, $\mu_i^B ( \cdot\ |[a_i]) = \overline \mu_i^B$. Also, this measure is a probability measure, since
\begin{eqnarray*}
\mu_i^B(X) &=&\mu [a_i] \int_{[a_i]} \vp_{a_i} d\overline \mu_i^B\\
&=&\mu [a_i] \sum_{B\in \beta} \overline \mu_i^B (B) \cdot length(B) \mbox{ (The return time is constant on partition sets)}\\
&=& \mu [a_i] \sum_{B\in \beta} \overline \mu_i (B) \cdot length(B) = \mu [a_i] \int_X \vp_{a_i} d\overline \mu_i =1,
\end{eqnarray*}
using Kac's formula and the definition of $\overline \mu_i$. Now, $\overline \mu_i^B$ is $\overline T$ - ergodic (as a Bernoulli measure). If $f$ is a measurable $\mu_i^B$ invariant function, then $f|_{[a_i]}$ is $\overline T$ invariant and so it is constant. Since $[a_i]$ is sweep-out, $f$ is constant a.e. This implies that $\mu_i^B$ is $T$ -ergodic (we'll use this observation later).
We now proceed to show that $\phi \in  L^1(\mu_i^B)$. Set $M:= \sup_{n\geq 2} \vari_{n+1} \phi_n$ (by the Walters property of $\phi$) and define $\overline \phi:= \sum_{j=0}^{\vp_{a_i} -1} \phi \circ T^j$.\\
By the definition of the partition $\beta$, partition sets are cylinders. We also see from the same definition that $length(B)-1 = \vp_{[a_i]}(x)$ for any $x\in B$. So $\vp_{[a_i]}$ is constant on partition. Let $B\in \beta$ and fix $x_B \in B$. For any $y\in B$ we get that
\begin{eqnarray*}
|\ |\overline \phi (x_B) | - |\overline \phi (y)| \ | &\leq & |\overline \phi (x_{B}) -\overline \phi (y)|\\
&\leq & \vari_{length(B)+1}\phi_{length(B)} \leq M  \text{ (since $length(B) \geq 2$) },
\end{eqnarray*}
so $\biggl |\ |\overline \phi (x)| - |\overline \phi (y)|\ \biggr| \leq 2M$ for every $x,y \in B$ and every $B\in \beta$. Since $\overline \mu_i^B (B) = \overline \mu_i (B)$, $\overline \phi\in L^1(\overline \mu_i^B) \Leftrightarrow \overline \phi\in L^1(\overline \mu_i)$.
Since
\begin{eqnarray*}
\int |\overline \phi | d\overline \mu_i &\leq & \int \overline{|\phi |} d\overline \mu_i\\
&=& \frac{1}{\mu_i [a_i]} \int  | \phi |d \mu_i \text{ (Kac's formula)}\\
&\leq & \frac{1}{\mu_i [a_i]\cdot \mu (E_i)}\int |\phi | d\mu < \infty \text{ (First claim)},\\
\end{eqnarray*}
we see that $\overline \phi \in L^1(\overline \mu_i)$ which is equivalent to $\overline \phi \in L^1(\overline \mu_i^B)$. This means that $\phi \in L^1 (\mu_i^B)$.
Now, that last fact implies that $h_{\mu_i^B}(T) + \int \phi d\mu_i^B$ is well defined, so by the variational principle - $h_{\mu_i^B}(T) + \int \phi d\mu_i^B \leq 0 = \pg$. We get that  $h_{\mu_i^B}(T) \leq -\int \phi d\mu_i^B < \infty$. Abramov's formula now shows (recall $[a_i]$ is a sweep-out set) that $h_{\overline \mu_i^B} (\overline T) = \frac{1}{\mu_i^B [a_i]} h_{\mu_i^B}(T) <\infty$ and we are done.
\end{proof}
Proceed with the last claim we need to complete the proof of the uniqueness theorem.
\begin{claim}
$\frac{d\mu_i}{d\mu_i \circ T} =g\ \mu_i$-a.e. for all $i$.
\end{claim}
\begin{proof}
Let $g_i := \frac{d\mu_i}{d\mu_i \circ T}$. Let $\widehat T_{\mu_i}$ be the transfer operator of $\mu_i$. We have $\widehat T_{\mu_i}f = \sum_{Ty=x}g_i(y)f(y)$. By the properties of the transfer operator, (proposition \ref{ruelle properties}), $\widehat T_{\mu_i} 1$ is the unique $L^1$ element s.t. $\forall \vp\in L^{\infty}$ the following holds: $\int \vp \widehat T_{\mu_i}1 d\mu_i = \int \vp \circ T d\mu_i = \int \vp d\mu_i$ by $T$'s $\mu_i$ invariance. Thus, $\mu_i$ a.e $\widehat T_{\mu_i}1 =1$. This implies that $g_i$ is a g-function. Now the construction from the first step shows that for almost every ergodic component $\int \vp - \vp \circ T d\mu_x =0$. By definition of $\mu_i$ (and the fact that we discarded all $i$'s for which $\mu (E_i) = 0$) we get that $\forall i,\ \int (\vp - \vp \circ T) d\mu_i =0$. Thus:
\begin{eqnarray*}
0&=& \ent{\mu_i} + \int \phi d\mu_i \text{ ($\mu_i$ is an equilibrium measure) }\\
&=&\ent{\mu_i} + \int \log g d\mu_i \text{ (cohomology and claim)}\\
&=&\int \log \frac{g}{g_i} d\mu_i \text{ (previous claim)}\\
&=&\int \widehat T_{\mu_i} \log \frac{g}{g_i} d\mu_i \text{ (since $\widehat T_{\mu_i}^*\mu_i = \mu_i$)}\\
&=&\int \biggl ( \sum_{Ty=x} g_i(y) \log \frac{g(y)}{g_i(y)} \biggr) d\mu_i(x) =: (*) \text{ (by definition) }.\\
\end{eqnarray*}
The term in brackets is defined for $\mu_i$-a.e. $x$, so if there exists some set $A$ with $\mu_i (A)>0$ s.t. for every $x\in A$ there exists $y\in T^{-1}\{x\}$ with $g_i(y) < 0$, the above term would be undefined on a set of positive $\mu_i$ measure, a contradiction. So for our purposes, $g_i(y) \geq 0$. We sum over those $y$'s for which $g_i(y) > 0$. This does not change the sum, since we neglect only $y$'s for which $g_i(y)=0$ (and agree that $0\log 0 = 0\log\infty = 0$). This allows us to use Jensen's inequality (recalling that $\log$ is concave):
\begin{eqnarray*}
(*) &=& \int ( \sum_{Ty=x, g_i(y)>0} g_i(y) \log \frac{g(y)}{g_i(y)} ) d\mu_i(x)\\
&\leq & \int \log \sum_{Ty=x, g_i(y) >0} g(y) d\mu_i(x) \\
&\leq & 0 \text{ (since $g$ is a sub g-function)}\\
\end{eqnarray*}
and the inequalities are actually equalities. Using Jensen's inequality again, recalling $g_i(y)$ is a $g$-function and $g$ is a sub-$g$-function, we get that $\sum_{Ty=x, g_i(y)>0}g_i(y)\log\frac{g(y)}{g_i(y)} \leq \log \sum_{Ty=x,g_i(y)>0} g(y) \leq 0$. This implies that for $\mu_i$ a.e.-$x$, $\forall y\in T^{-1}\{x\}$ $c(x)g_i(y)=g(y)$. Thus,
$$
0=\sum_{Ty =x,g_i(y)>0} g_i(y) \log \frac{g(y)}{g_i(y)} =\log c(x) \sum_{Ty =x,g_i(y)>0} g_i(y) =\log c(x) \sum_{Ty=x}g_i(y)=\log c(x) \text{ $\mu_i$-a.e. $x$},
$$
since $g_i$ is a g-function. So $c(x) = 1$ for $\mu$ a.e.-$x$ and $g = g_i$ $\mu_i$ a.e..
\end{proof}
We now complete the proof. First, notice that by assumption, $\lambda = \exp \pg = e^0 =1$ and so by last claim $L^*_{\log g} \mu_i=\mu_i$. The following holds by the cohomology relation from the beginning.
\begin{eqnarray*}
\int f e^{\vp} d\mu_i &=& \int f e^{\vp} d L_{\log g}^* \mu_i = \int L_{\log g}(fe^{\vp}) d\mu_i \\
&=& \int \sum_{Ty=x} \exp(\phi - \vp + \vp \circ T)(y) (f\cdot e^{\vp}) (y) d\mu_i (x)\\
&=& \int e^{\vp (x)} \sum_{Ty=x} e^{\phi (y)} f(y) d\mu_i (x) = \int (L_{\phi} f)\cdot e^{\vp} d\mu_i = \int f d(L_{\phi}^* e^{\vp} \mu_i)\\.
\end{eqnarray*}
So we got $L_{\phi}^*(e^{\vp}\mu_i)=e^{\vp}\mu_i$ for every $i$. Since $\mu = \sum p_i\mu_i$, this holds for $\mu$: $L_{\phi}^*(e^{\vp}\mu)=e^{\vp}\mu$. Now we show $L_{\phi}e^{-\vp} = e^{-\vp}$. We saw already that $L_{\log g}1 =1$ $\mu_i$ a.e. for every $i$, hence $L_{\log g}1 =1$ $\mu$ a.e. Now
\begin{eqnarray*}
1= L_{\log g} 1 &=& \sum_{Ty =x} e^{\log g(y)} = \sum_{Ty=x} \exp (\phi -\vp +\vp \circ T)\\
&=& e^{\vp}\sum_{Ty=x}e^{\phi (y)} e^{-\vp}(y) = e^{\vp} L_{\phi}e^{-\vp},\\
\end{eqnarray*}
so $L_{\phi} e^{-\vp} = e^{-\vp} \ \mu$-a.e. We now need the following fact:
\begin{fact}
\ass Let $\nu$ be such that $L_{\phi}^*\nu = \lambda \nu$ with $\lambda >0$. Then for any cylinder $[\underline a]$, $\nu [\underline a] >0$.
\end{fact}
\begin{proof}
Let $n :=|\underline a |$. Fix $p\in S$ and denote $N:= N_{a_np}$ the length of the path $a_n \to p$ we get from the topological mixing property.  Then $\nu [\underline a] \geq \int_{[p]} 1_{[\underline a]}(x)d\nu (x) =\lambda^{-(N+n)} \int_{[p]} L_{\phi}^{N+n} 1_{[\underline a]} d\nu = \lambda^{-(N+n)} \int_{[p]} \sum_{T^{N+n}y=x} e^{\phi_{N+n} (y)}1_{[\underline a]}(y) d\nu (x)$.\\
We show that $\forall x \in [p], \sum_{T^{N+n}y=x} e^{\phi_{N+n} (y)}1_{[\underline a]}(y) >0$. The exponent is always positive, so it can be ignored. We ask whether for every $x\in [p]$ there exists $y\in T^{-(N+n)}\{x\}$ and $y_{0}^{n} = \ua$. This is true since after $N$ preimages of $x$ are taken, there has to be one preimage with a prefix $(a_n,...,p)$. Taking further $n$ preimages guarantees one of them will be have prefix $(\ua,p)$. For that preimage, $1_{[\ua]} =1$ and at least one summand is positive. Hence $\sum_{T^{N+n}y=x} e^{\phi_{N+n} (y)}1_{[\underline a]}(y) > 0$ on $[p]$ and we are done.
\end{proof}
$L_{\phi}^* e^{\vp} \mu = e^{\vp} \mu$ implies that for every word $\underline a$ we have $\int_{[\underline a]} e^{\vp} d\mu >0$, hence $\mu[\ua] > 0$ for every cylinder $[\ua]$. A property that holds a.e. for a measure that is positive on open sets necessarily holds on a dense set, so $L_{\phi}e^{-\vp}= e^{-\vp}$ on a dense set. Continuity of $\vp$ (theorem \ref{cohomology g functions}) implies equality holds everywhere. By the discussion at the beginning, we are done.
\end{proof}
We finish with a corollary which will be importatnt in the next section.
\begin{corollary}\label{corollary2uniqueness}
\ass Let $\sup \phi < \infty, \pg<\infty$ (as in the theorem) and let $\mu$ be the \eqm of $\phi$. Then for every finite sub-alphabet $S^* \subseteq S$ there exists a constat $C^* = C^*(S^*) > 1$ s.t. for every $m,n \geq 1$, every n-cylinder $[\underline a]$ and every m-cylinder $[\underline c]$ the following hold:
\begin{enumerate}
\item if the last letter in $\underline a$ is in $S^*$ and $[\underline a,\underline c] \not = \emptyset$ then $\frac{\mu [\underline a,\underline c]}{\mu [\underline a]\mu [\underline c]} = (C^*)^{\pm 1}$.
\item if the first letter in $\underline a$ is in $S^*$ and $[\underline c,\underline a] \not = \emptyset$ then $\frac{\mu [\underline c,\underline a]}{\mu [\underline a]\mu [\underline c]} = (C^*)^{\pm 1}$.
\end{enumerate}
\end{corollary}
\begin{proof}
Since $\mu$ is an equilibrium measure, it is an RPF measure and $\phi$ is positive recurrent. Thus the convergence $\lambda^{-n} \Lp^n 1_{[a]} \rightarrow \nu [a]h$ holds (by GRPF, theorem \ref{GRPFtheorem}). Let $\phi^*:= \phi + \log h - \log h \circ T +\log \lambda$. First, wlog we assume $\lambda = \exp(\pg) = 1$ (equivalently, $\pg= 0$). We show $\phi^*$ is Walters. Then the corollary follows verbatim as in \cite{sar2} and we omit the proof.
We need to show $\sup_{n\geq 1}\vari_{n+m}\phi^*_n \rightarrow 0$ and $\sup_{n\geq 1}\vari_{n+m}\phi^*_n <\infty$ (recall $\phi^*_n:= \phi^*+...+\phi^*\circ T^{n-1}$). We begin with the first of these conditions.
\begin{claim}
$\vari_{n+m} \phi^*_n \rightarrow 0$ as $m\rightarrow \infty$.
\end{claim}
\begin{proof}
Since $\phi$ is Walters, we need to show $\vari_{n+m} [\log h - \log h \circ T]_n \rightarrow 0$. We do that. $\vari_{n+m} [\log h - \log h \circ T]_n = \vari_{n+m} [\log h - \log h \circ T^n] \leq 2 \vari_m \log h.$
We show $\vari_{m}\log h \rightarrow 0$. Instead of analyzing the difference $\log h(x) - \log h(y)$, we analyze $\log\frac{h(x)}{h(y)}$. Recall that $h(x) =\nu[a] \lim_{n\rightarrow \infty} \Lp^n 1_{[a]}(x) =\lim_{n\rightarrow \infty}  \sum_{\underline p} \exp \phi_n (a,\underline p,x_0^{\infty})$, where $| \underline p | = n-1$, $(a,\underline p,x_0^{\infty})$ is admissible.
\begin{eqnarray*}
\frac{h(x)}{h(y)} &=& \lim_{n\rightarrow \infty}  \sum_{\underline p} \exp \phi_n (a,\underline p,x_0^{\infty}) \bigl / \lim_{n\rightarrow \infty}  \sum_{\underline p} \exp \phi_n (a,\underline p,y_0^{\infty})\\
&=& \lim_{n\rightarrow \infty}  \sum_{\underline p} \exp \phi_n (a,\underline p,x_0^{\infty}) / \sum_{\underline p} \exp \phi_n (a,\underline p,y_0^{\infty})\\
&\leq & \lim_{n\rightarrow \infty}  \sum_{\underline p} \exp [\phi_n (a,\underline p,y_0^{\infty}) + \vari_{n+m}\phi_n] / \sum_{\underline p} \exp \phi_n (a,\underline p,y_0^{\infty})\text{ ($\because x_0^m = y_0^m$)}\\
&\leq & \exp\sup_{n\geq 1} \vari_{n+m} \phi_n\\
\end{eqnarray*}
Hence $\log \frac{h(x)}{h(y)} \leq \sup_{n\geq 1} \vari_{n+m} \phi_n \rightarrow 0$ by Walters property.
\end{proof}
\begin{claim}
$M:= \sup_{n\geq 1} \vari_{n+1} \phi^*_n < \infty$
\end{claim}
\begin{proof}
We now show that $\sup_{n \geq 1}\vari_{n+1} [\log h - \log h \circ T]_n = \sup_{n \geq 1} \vari_{n+1}[\log h - \log h \circ T^n]\leq 2\vari_{1} [\log h] < \infty$. Recall that (up to a multiplicative constant), $h(x) = \lim_{m\rightarrow \infty} \Lp^m 1_{[a]}(x)$. Assume now that $x_0 = y_0=b$ and denote $K:=\max\{1, \sup_{n\geq 1}\vari_{n+1} \phi_n\}$ (finite, by Walters property):
\begin{eqnarray*}
|\log h(x) - \log h(y)| &=& |\log \lim_{n\rightarrow \infty} \Lp^n 1_{[a]}(x) -\log \lim_{n\rightarrow \infty}\Lp^n 1_{[a]}(y)|\\
&=&|\lim_{n\rightarrow \infty}\log \Lp^n 1_{[a]}(x) - \lim_{n\rightarrow \infty}\log \Lp^n 1_{[a]}(y)| \text{ (log is continuous on its domain)}\\
&=& \left\vert \lim_{n\rightarrow \infty} \log \frac{ \sum_{(a,\underline p,b,x_{1}^{\infty}),|\underline p|= n-1  } \exp\phi_n(a,\underline p,b,x_{1}^{\infty})} {\sum_{(a,\underline p,b,y_{1}^{\infty}),  |\underline p|=n-1} \exp\phi_n(a,\underline p,b,y_1^{\infty})} \right \vert\ \\
&\leq & \left\vert \lim_{n\rightarrow \infty} \log \frac{ \sum_{(a,\underline p,b,x_{1}^{\infty}), |\underline p| =n-1 } \exp\phi_n(a,\underline p,b,x_{1}^{\infty})} {\sum_{(a,\underline p,b,x_{1}^{\infty}), |\underline p| =n-1} \exp\phi_n(a,\underline p,b,x_1^{\infty}) \cdot \exp(-K)} \right \vert\ \leq \exp (K).\\
\end{eqnarray*}
We need the assumption $x_0 = y_0$ since this allows us to sum over the same $\underline p$'s. Had it not been the case, we'd have no control over the number of summands and the above argument would not have held.\\
Denote $M':= \sup_{n\geq 1} \vari_{n+1} \phi_n$ and recall $[\log h - \log h \circ T]_n$ is telescopic. We get:
\begin{eqnarray*}
\vari_{n+1}\phi^*_n &\leq & \vari_{n+1} \phi_n + \vari_{n+1} [\log h - \log h \circ T^n]\\
&\leq & M' +2\vari_{1}\log h <M' + 2\exp (K) < \infty\\
\end{eqnarray*}
\end{proof}
These two claims show that $\phi^*$ is Walters - the relevant variations are all bounded and approach zero (using the first claim and the Walters property of $\phi$).
\end{proof}
\section{Ornstein theory} \label{Ornstein}
In this section we finish the proof of theorem \ref{reduced theorem}, which implies theorem \ref{main theorem} by the discussion in section \ref{transection}. On two sided shift spaces one needs to specify where a cylinder begins (e.g. for $\ua = (a_0,...,a_{n-1})$ we write $\ _m[\ua] =\{ \hat x\in \hat X : x_m^{m+n} = \ua\}$, the cylinder of $\ua$ starting at the $m$th coordinate). In order to make notation easier, we assume that if a starting coordinate is not specified, it is zero (i.e. $[\ua] = \{\hat x \in \hat X : x_0^{n} = \ua\}$).\\
Let $m<n$ integers. For a partition $\beta$, denote $\beta_{m}^n : = \bigvee_{i=m}^{n} T^{-i}\beta$. A sequence of partitions $\{\beta_m\}_{m=1}^{\infty}$ is called increasing if $\beta_{m+1} \vee \beta_{m} = \beta_{m+1}$. It is generating if $\sigma\left (\bigvee_{i=0}^{n}\beta_i \right ) \uparrow \mathcal{B}$ (the Borel $\sigma$-algebra).  
\begin{definition}
A finite measurable partition $\beta$ is called weakly Bernoulli if $\forall  \epsilon >0\ \exists k > 1$ s.t. $\sum_{A\in \beta_{-n}^0} \sum_{B\in \beta_{k}^{n+k}} |\hat\mu(A\cap B) - \hat \mu(A)\hat \mu(B)| < \epsilon$ for all $n>0$.
\end{definition}
By the work of Ornstein \& Friedman, \cite{ornstein} , we know that if an invertible probability preserving transformation has a generating sequence of weak Bernoulli partitions, then it is measure theoretically isomorphic to a Bernoulli scheme. This is the heart of the proof of theorem \ref{reduced theorem} (consequently, theorem \ref{main theorem}) and the focus of this section.\\

In order to use the above notions, we would like to use the information we gathered on one-sided shifts and apply it to two-sided shifts. We now explain how this is done. Assume $\hat \phi \in C(\hat X), \phi \in C(X)$ satisfy $\hat \phi= \phi\circ \pi +h - h\circ T$, where $h\in C(\hat X)$ is bounded and $\pi$ is the natural projection (see section \ref{waltersection}).\\
It is known from \cite{roh} that if $\hat \mu$ is a $\hat T$-invariant measure, it induces a $T$-invariant measure $\mu$ on the natural extension with the same entropy ($h_{\hat \mu}(T) = h_{\hat \mu}(\hat T)$). Consequently, a $\hat T$-invariant measure on the two-sided shift induces a $T$-invariant measure on the one-sided shift with the same entropy. Let $\hat \mu$ be $\hat T$-invariant and let $\mu$ be the $T$-invariant measure it induces. $\int \hat \phi d\hat \mu = \int \phi \circ \pi +h-h\circ T d\hat \mu = \int \phi \circ \pi d\hat \mu = \int \phi d\mu$. Consequently, $h_{\hat \mu}(\hat T) + \int \hat \phi d\hat \mu = h_{\mu}(T) + \int \phi d \mu$ (we rely on the fact that the transfer function we constructed in the proof of theorem \ref{waltercohomology}, here denoted $h$, is bounded).\\
The above implies $P_G(\hat \phi) = P_G(\phi)$ and so the equilibrium measure for the one-sided shift with the potential $\phi$ is the projection of the two-sided equilibrium measure of the potential $\hat \phi$. This means that any bound we find for equilibrium measures of cylinders in the one-sided shift space automatically applies to the equilibrium measure of $\hat \phi$. Particularly, we may use corollary \ref{corollary2uniqueness}. Using the assumptions and definitions there, assuming wlog that $\pg = 0$ we have, $\forall F\in L^1(\mu)$:
\begin{eqnarray*}
\int \Lps F d\mu &=& \int \sum_{Ty=x} \exp (\phi(y)) \frac{h(y)}{h\circ T (y)}F(y) d\mu = \int \frac{1}{h} \Lp (hf)hd\nu\\
&=& \int \Lp (hF) d\nu = \int F hd\nu  = \int F d\mu ,\\
\end{eqnarray*}
so $\Lps ^*\mu = \mu$. Now it is easy to derive the following:
$$
f,g \in C(X), f,g >0 \Longrightarrow  \int f(g\circ T^n)d\mu = \int (\Lps^n f)g d\mu .
$$
Before we prove theorem \ref{reduced theorem}, let us first introduce a useful notation: $a = M^{\pm 1}b$ iff $M^{-1}b \leq a \leq Mb$.
\begin{proof}[Proof of theorem \ref{reduced theorem}.]
Suppose $\hat \mu$ is an \eqm of $\hat \phi\in C(\hat X)$ as in the statement of the theorem. For every finite set  of letters (states) $\cal{V}' \subseteq \cal{V}$, we let $\alpha(\cal{V}'):= \{\ _0[v]: v\in \cal{V}'\} \cup \{ \bigcup_{v\not \in \cal{V}'}\ _0[v]\}$. We show that $\alpha(\cal{V}')$ is weak Bernoulli. This implies the theorem, by the discussion above.\\
As in \cite{sar2}, we make some choice of parameters. The theorem holds, given that we are able to choose these parameters. We first fix some small $\delta_0 >0$ s.t. every $0<t<\delta_0$ satisfies $1- e^{-t} \in (\frac{1}{2}t,t)$. We fix some smaller $\delta < \delta_0$ to be determined later. Then we choose:
\begin{itemize}
\item Some finite collection $S^* \subseteq S$ of states s.t. $\hat \mu (\cup_{a\in S^*} [a])>1-\delta$,
\item A constant $C^*= C^*(S^*) > 1$ as in corollary \ref{corollary2uniqueness},
\item $m\in \bb{N}, m = m(\delta)$ s.t. $\sup_{n\geq 1} \vari_{n+m}\phi^{*}_n < \delta$. This can be done since $\phi^*$ is Walters,
\item A finite collection $\gamma$ of $m$-cylinders $[\underline c]$ s.t. $\hat \mu (\cup \gamma)> e^{-\delta / 2(C^*)^2}$,
\item Points $x(\underline c) \in [\underline c] \in \gamma$,
\item Natural numbers $K(\underline c,\underline c')$ where $[\underline c], [\underline c'] \in \gamma$ s.t. for every $k\geq K(\underline c,\underline c')$,
$$
\Lps^k 1_{[\underline c]}(x(\underline c')) =e^{\pm \delta} \mu[\underline c].
$$
\item $K(\delta) := \max \{K(\underline c,\underline c'): \underline c', [\underline c]\in \gamma\} +m$.
\end{itemize}
Given these choices of parameters, we can complete the proof. This proceeds very similarly to  \cite{sar2}. We start with a claim.
\begin{claim}
Let $A:=\ _{-n}[a_0,...,a_n], B:=\ _k[b_0,...,b_n]$ be two non-empty cylinders of length $n+1$. Let $b_0,a_n \in S^*$. Then for every $k >K(\delta)$ and every $n\geq 0$,
$$
|\hat \mu (A\cap B) -\hat \mu(A)\hat \mu(B)| <15\delta\hat \mu(A)\hat \mu(B).
$$
\end{claim}
\begin{proof}
We denote $\alpha_m$ the collection of $m$-cylinders $[\underline c]$. For every $k > 2m$,
\begin{eqnarray*}
\hat \mu(A\cap B) &=& \sum_{[\underline c], [\underline c'] \in \alpha_m} \hat \mu(_{-n}[\underline a, \underline c] \cap\ _{k-m}[\underline c', \underline b])\\
&=& \sum_{[\underline c], [\underline c'] \in \alpha_m} \hat \mu ([\underline a,\underline c] \cap T^{-n}\ _{k-m}[\underline c', \underline b])\ \text{(shift invariance)}\\
&=& \sum_{[\underline c], [\underline c'] \in \alpha_m} \mu([\underline a,\underline c] \cap T^{-n-(k-m)} [\underline c', \underline b])\\
&=& \sum_{[\underline c], [\underline c'] \in \alpha_m} \int 1_{[\underline a,\underline c]} \cdot 1_{[\underline c', \underline b]}\circ T^{(n+k-m)}d\mu\\
&=& \sum_{[\underline c], [\underline c'] \in \alpha_m} \int \Lps^{n+k-m}1_{[\underline a,\underline c]} \cdot 1_{[\underline c', \underline b]}d\mu\\
&=& \sum_{[\underline c],[\underline c']\in \gamma}\int_{[\underline c', \underline b]} \Lps^{n+k-m}1_{[\underline a,\underline c]}d\mu   + \sum_{[\underline c] \in\alpha_m \backslash \gamma \text{ or }[\underline c']\in\alpha_m \backslash \gamma} \int_{[\underline c', \underline b} \Lps^{n+k-m}1_{[\underline a,\underline c]}d\mu.\\
\end{eqnarray*}
We proceed and estimate the terms. The left is called the main term, the second will be called the error term. We start with a preliminary estimate we will use throughout.
\subsection*{Preliminary estimate}
We use the following fact: $\forall y \in [\underline c' , \underline b]$,
\begin{eqnarray*}
\Lps^{k+n-m}1_{[\underline a,\underline c]}(y) &=& \sum_{T^{n+k-m}z=y}\exp (\phi^*_{k+n-m}(z)) 1_{[\underline a,\underline c]}(z)\\
&=& \sum_{T^{n+k-m}z=y} \exp\left( \sum_{i=0}^{k+n-m-1} \phi^* ( T^iz)\right ) 1_{[\underline a,\underline c]}(z)\\
&=& \sum_{T^{n+k-m}z=y} \exp\left( \sum_{i=0}^{n} \phi^* (T^iz)\right)\exp\left(\sum_{i=n+1}^{k+n-m-1} \phi^* (T^iz) \right )1_{[\underline a,\underline c]}(z)\\
&=& \sum_{T^{n+k-m}z=y} \exp( \phi^*_{n+1}(z))\exp( \phi^*_{k-m-1} (T^{n+1}z))1_{[\underline a,\underline c]}(z)\\
&=& \sum_{T^{k-m-1}z=y} \exp( \phi^*_{n+1}(\underline a,z))\exp( \phi^*_{k-m-1}(z))1_{[\underline c]}(z).\\
\end{eqnarray*}
By our choice of $m(\delta)$ before, we can estimate $\exp(\phi^*_{n+1}(\underline a,z)) =e^{\pm \delta}\exp(\phi^*_{n+1}(\underline a,w))$ for every $w,z \in [\underline c]$. Fix $z$ and average over all $w$'s:
\begin{eqnarray*}
\exp(\phi^*_{n+1}(\underline a,z)) &=& \frac{e^{\pm \delta}}{\mu[\underline c]}\int_{[\underline c]} \exp(\phi^*_{n+1}(\underline a,w))d\hat \mu(w)\\
&=& \frac{e^{\pm \delta}}{\mu[\underline c]}\int \Lps^{n+1}1_{[\underline a,\underline c]}d\mu = e^{\pm \delta}\frac{\mu[\underline a,\underline c]}{\mu[\underline c]}\\
\end{eqnarray*}
The reason we did the substitution with $\Lps$ is that there's actually just one summand in the expression for $\Lps^{n+1}1_{[\underline a,\underline c]}$.
Using the previous two estimates, we can get:
\begin{equation}\label{master}
\Lps^{k+n-m}1_{[\underline a,\underline c]}(y)= e^{\pm \delta}\frac{\mu[\underline a,\underline c]}{\mu[\underline c]} \Lps^{k-m-1}1_{[\underline c]}(y) \text{ for every $y \in [\underline c', \underline b]$}.\\
\end{equation}
We refer to this as the master estimate.

\subsection*{Main term}
First, we assume $k > K(\delta)$ and $k>2m(\delta)$. Since we estimate the main term, we know $[\underline c],[\underline c'] \in \gamma$. Note also $y \in [\underline c', \underline b]$.
\begin{equation*}
\Lps^{k-m-1}1_{[\underline c]}(y) = e^{\pm \delta}\Lps^{k-m-1}1_{[\underline c]}(x(\underline c'))= e^{\pm 2\delta}\mu[\underline c'].\\
\end{equation*}
Plugging into the master estimate (above, equation \ref{master}) we get that $k > K(\delta), k>2m(\delta)$ implies $\Lps^{n+k-m}1_{[\underline a,\underline c]}(y) = e^{(\pm 3\delta)} \mu[\underline a,\underline c]$ on $[\underline c', \underline b]$.
We can integrate and see the main term equals:
\begin{equation*}
\sum_{[\underline c]\in \gamma ,[\underline c']\in \gamma}\int_{[\underline c', \underline b]} \Lps^{n+k-m}1_{[\underline a,\underline c]}d\mu =
e^{\pm 3\delta} \sum_{[\underline c]\in \gamma ,[\underline c']\in \gamma} \mu [\underline a,\underline c] \mu[\underline c',\underline b] =
e^{\pm 3\delta}  (\sum_{[\underline c]\in \gamma} \mu [\underline a,\underline c]) (\sum_{[\underline c']\in \gamma} \mu [\underline c',\underline b])\\
\end{equation*}
The first bracketed term is bounded from above by $\mu[\underline a]$. To bound below we use corollary \ref{corollary2uniqueness}:
\begin{eqnarray*}
\sum_{[\underline c] \in \gamma, [\underline a,\underline c] \not = \emptyset}\mu[\underline a,\underline c] &=& \mu[\underline a] - \sum_{[\underline c] \in\alpha_m \backslash \gamma, [\underline a,\underline c] \not = \emptyset} \mu[\underline a,\underline c] \geq \mu[\underline a] (1 - \sum_{[\underline c] \in\alpha_m \backslash \gamma, [\underline a,\underline c] \not = \emptyset} C^* \mu[\underline c])\\
&\geq & \mu[\underline a] (1 - \sum_{[\underline c] \in\alpha_m \backslash \gamma} C^* \mu[\underline c])\\
&\geq & \mu[\underline a](1 - C^*(1- e^{-\delta /2(C^*)^2})) \text{ by choice of $\gamma$}\\
&\geq & \mu[\underline a] e^{-\delta}, \text{ by choice of $\delta_0$ }
\end{eqnarray*}
So the first sum equals $e^{\pm \delta} \mu[\underline a]$. Similarly, the second is $e^{\pm \delta} \mu[\underline b]$ and we get that the main term equals $e^{\pm 5\delta}\mu[\underline a]\mu[\underline b] = e^{\pm 5\delta}\hat \mu(A)\hat \mu(B)$.
\subsection*{Error term}
Since $a_n \in S^*$ we may use the master estimate (equation \ref{master}) in conjunction with corollary \ref{corollary2uniqueness}:
$$
\Lps^{n+k-m}1_{[\underline a,\underline c]}(y) \leq C^*e^{\delta} \mu[\underline a]\Lps^{k-m-1}1_{[\underline c]}(y) \text{ on $[\underline c',\underline b]$}
$$
therefore
\begin{eqnarray*}
\text{Error term } &\leq & C^*e^{\delta} \mu[\underline a] \sum_{[\underline c]\in \alpha_m \backslash\gamma \text{ or } [\underline c']\in\alpha_m \backslash \gamma} \int_{[\underline c', \underline b]} \Lps^{k-m-1}1_{[\underline c]}d\mu\\
&=& C^*e^{\delta} \mu[\underline a] \sum_{[\underline c]\in \alpha_m \backslash\gamma \text{ or } [\underline c']\in\alpha_m \backslash \gamma}
\mu([\underline c] \cap T^{-(k-m-1)}[\underline c', \underline b])\\
&\leq & C^*e^{\delta} \mu[\underline a] \left (\sum_{[\underline c] \in \alpha_m \backslash \gamma} \mu([\underline c] \cap T^{-(k-1)}[\underline b]) + \sum_{[\underline c'] \in \alpha_m \backslash \gamma} \mu(T^{-(k-m-1)}[\underline c',\underline b]) \right )\\
&=& C^*e^{\delta} \mu[\underline a] \left (\sum_{[\underline d] \in \alpha_{k-1},[d_0,...,d_{m-1}]\not \in \gamma} \mu [\underline d,\underline b] + \sum_{[\underline c'] \in \alpha_m \backslash \gamma} \mu[\underline c',\underline b] \right )\\
&\leq & (C^*)^2e^{\delta} \mu[\underline a] \left ( \sum_{[\underline d] \in \alpha_{k-1},[d_0,...,d_{m-1}]\not \in \gamma} 
\mu[\underline d]\mu[\underline b] + \sum_{[\underline c'] \in \alpha_m \backslash \gamma} \mu[\underline c']\mu[\underline b] \right ) \text{ ($\because b_0 \in S^*$)}\\
&\leq & (C^*)^2e^{\delta} \mu[\underline a]\mu[\underline b]\cdot 2\hat \mu (\cup \gamma)^c\\
&\leq & e^{\delta} \delta \mu[\underline a]\mu[\underline b] \text{ by choice of $\gamma$}\\
&<& 5\delta \mu[\underline a]\mu[\underline b].\\
\end{eqnarray*}
So the error term is less than $5\delta \hat \mu(A)\hat \mu(B)$.\\
Adding these two estimates together we see that 
$$
(-5\delta + e^{-5\delta}-1)\hat \mu(A)\hat \mu(B) \leq \hat \mu(A\cap B) - \hat \mu(A)\hat \mu(B) \leq (e^{5\delta}+ 5\delta -1) \hat \mu(A)\hat \mu(B).
$$
This translates to:
\begin{equation*}\tag{*}
|\hat \mu(A\cap B) - \hat \mu(A)\hat \mu(B)| \leq \hat \mu(A)\hat \mu(B) \max\{e^{5\delta} + 5\delta -1, 1- e^{-5\delta} + 5\delta\} < 15\delta \hat \mu(A)\hat \mu(B),
\end{equation*}
for $\delta$ sufficiently small.\\
\end{proof}
\begin{claim}
Let $k = K(\delta)$, where $\delta$ was chosen as before. Let $n \geq 0$. Then:
$$
\sum_{A\in \alpha_{-n}^0, B \in \alpha_k^{k+n}} |\hat \mu (A\cap B) - \hat \mu(A)\hat \mu(B)| <19\delta.
$$
\end{claim}
\begin{proof}
Write $A:= _{-n}[a_0,...,a_n], B:= _k[b_0,...,b_n]$. We break the sum into
\begin{itemize}
\item The sum over $A,B$ for which $a_n,b_0 \in S^*$.
\item The sum over $a_n \not \in S^*$.
\item The sum over $a_n \in S^*, b_0 \not \in S^*$.
\end{itemize}
The first sum is bounded (using previous claim) by $15\delta$. The second and third are bounded (each) by $2\hat \mu (\cup_{a \not \in S^*} [a]) < 2\delta$. The claim follows.
\end{proof}
To conclude, we show that $\alpha(\cal{V}')$ is weakly Bernoulli. To do that, let $\epsilon > 0$. Choose $\delta>0$ so small that both $19\delta <\epsilon$ and (*) holds. Noting that the partitions $\alpha(\mathcal{V}')$ are coarser than $\alpha(\mathcal{V})$, the weak Bernoullicity for $\alpha(\mathcal{V}')$ follows, using the triangle inequality.
\end{proof}
\section{Cohomology- summable variations} \label{coelhosection}
\begin{theorem}\label{coelhotheorem}
Suppose $f: \hat X \to \bb{R}$ has summable variations with finite first variation (i.e. $\sum_{n=1}^{\infty} \vari_n f < \infty$). Then there exists a $g$ with summable variations that depends only on its non-negative coordinates which is cohomologous to $f$ via a bounded continuous transfer function.
\end{theorem}
\begin{proof}
We follow \cite{coel}, who proved this in the case of finite alphabet. Since $\vari_n f \to 0$, there exists, for every $i\geq 0$ large enough, $n_i$ s.t. $\vari_{n_i}f\leq 2^{-i}$ and $\vari_{n_i-1}f > 2^{-i}$. Let $i':= \min\{i\in\mathbb{N}: \exists n_i\ s.t.\ \vari_{n_i} f \leq 2^{-i}\ \wedge \vari_{n_i-1}f >2^{-i}\}$ and let $i_0:= \max \{i',2\}$. Set $n_i=0$  for $i <i_0$ . By summable variations,
\begin{eqnarray*}
\infty &>& \sum_{n=n_{i_0}}^{\infty} \vari_n f  > \sum_{i=i_0}^{\infty}(n_{i+1} - n_{i})2^{-(i+1)} = \lim_{I\to\infty}\sum_{i=i_0}^{I}(n_{i+1} - n_{i})2^{-(i+1)}\\
&=& \lim_{I\to\infty} \sum_{i= i_{0}+1}^{I}n_i 2^{-(i+1)} -n_{i_0}2^{-(i_0+1)} +n_{I+1}2^{-(I+1)}\\
&>& \frac{1}{2}\lim_{I\to \infty} \sum_{i=i_0}^{I} n_i 2^{-i} -n_{i_0}2^{-i_0}\\
\end{eqnarray*}
and we conclude that $\sum_{i= i_0}^{\infty}n_i 2^{-i} <\infty$.

Define for $a\in S: I_{a}:= \inf_{[a]}f$ (using the finiteness of $f$'s first variation). For $i\geq i_0$ define $f_i(x):= \inf_{[x_{-n_i+1}^{n_i-1}]}f(x)$ and $f_{i_0-1}(x)=I_{x_0}$. Note that $f_i \to f$ monotonically (trivial) and uniformly (by uniform choice of $i$ and uniform continuity of $f$). Now define $h_i = f_i - f_{i-1}$ for $i \geq i_0$.
$$
\Vert h_i \Vert_{\infty} = \Vert \inf_{[x_{-n_i+1}^{n_i-1}]}f - \inf_{[x_{-n_i+2}^{n_i-2}]}f \Vert_{\infty} \leq \vari_{n_i-1} f \leq 2^{-(i-1)}\ \forall i> i_0.
$$
Define also $g(x):= I_{x_0} + \sum_{i=i_0}^{\infty} h_{i}(T^{n_i-1}x)$. We use the following observations: 
\begin{itemize}
\item $g$ depends only on $x_{0}^{\infty}$,
\item $g$ is well defined and continuous,
\item $\forall n\ \vari_n(h_i \circ T^{n_i-1}) \leq 4\cdot 2^{-i}$, by the bound on $\Vert h_i\Vert_{\infty}$,
\item $\vari_n(h_i \circ T^{n_i-1})=0$ for $n>2n_i$, since $h_i$ is constant on cylinders from $-n_i+1$ to $n_i-1$.
\end{itemize}
Now we can show $g$ has summable variations:
\begin{eqnarray*}
\sum_{n=2}^{\infty}\vari_n (g)&=& \sum_{n=2}^{\infty}\vari_n \left (\sum_{i=i_0}^{\infty} h_i \circ T^{n_i-1} \right ) \leq  \sum_{i=i_0}^{\infty} \sum_{n=2}^{\infty} \vari_n ( h_i \circ T^{n_i-1})\\
&=& \sum_{i=i_0}^{\infty} \sum_{n=2}^{2n_i} \vari_n ( h_i \circ T^{n_i-1})\leq \sum_{i=i_0}^{\infty} 2n_i\cdot 4 \cdot 2^{-i} <\infty. \\
\end{eqnarray*}
Now find a transfer function. Define $F := \sum_{i=i_0}^{\infty}\sum_{j=0}^{n_i-2}h_i\circ T^j$. We show it is uniformly Cauchy. Let $F_N := \sum_{i=i_0}^{N}\sum_{j=0}^{n_i-1}h_i\circ T^j$. Let $N>M$.
\begin{equation*}
\Vert F_N- F_M \Vert_{\infty} = \Vert \sum_{i=M+1}^{N}\sum_{j=0}^{n_i-1}h_i\circ T^j \Vert_{\infty} \leq
\sum_{i=M+1}^{N}\sum_{j=0}^{n_i-1}\Vert h_i\Vert_{\infty} \leq \sum_{i=M+1}^{N}n_i2^{-(i-1)}
\end{equation*}
and this can be made arbitrarily small. Thus $F_N$ is a uniformly Cauchy series of uniformly continuous functions and $F$ is well defined and continuous.  A calculation shows this is indeed the transfer function. Since $F_N$ is a finite sum of terms that have bounded sup-norm, the limit $F = \lim_{N\to\infty} F_N$ is also bounded (uniform Cauchy). The function $g$ we have constructed is one sided so we are done.
\end{proof}
\section{Acknowledgements}
I would like to express my deepest gratitude to my M.Sc. advisor Omri Sarig for, simply put, being the best.
I would also like to thank the Weizmann Institute of Science and the Department of Mathematics and  Computer Science for giving me the best possible environment for conducting scientific research.
\bibliography{ref}
\bibliographystyle{amsalpha}
\end{document}